\documentclass[draft]{amsart}
\usepackage{cleveref}
\usepackage{autonum}
\usepackage{mathtools}
\usepackage{array}
\usepackage{bm}

\newtheorem{theorem}{Theorem}
\newtheorem{pretheorem}{Theorem}

\newtheorem{lemma}{Lemma}

\theoremstyle{definition}

\newtheorem{remark}{Remark}

\crefname{section}{Section}{Sections}

\crefname{theorem}{Theorem}{Theorems}
\crefname{pretheorem}{Theorem}{Theorems}
\crefname{corollary}{Corollary}{Corollaries}
\crefname{lemma}{Lemma}{Lemmas}
\crefname{definition}{Definition}{Definitions}
\crefname{remark}{Remark}{Remarks}
\crefname{example}{Example}{Examples}

\crefformat{equation}{(#2#1#3)}

\newcommand{\midmid}{\mathrel{}\middle|\mathrel{}}

\renewcommand{\epsilon}{\varepsilon}
\renewcommand{\phi}{\varphi}

\let\Re\relax
\DeclareMathOperator{\Re}{\mathrm{Re}}

\def\showsize#1{\setbox0=\hbox{#1}width=\the\wd0, height=\the\ht0, depth=\the\dp0}

\begin{document}

\title[On error term estimates \`a la Walfisz]
{On error term estimates \`a la Walfisz\\
for mean values of arithmetic functions}
\author[Y. Suzuki]{Yuta Suzuki}
\date{}


\maketitle

%%%%%%%%%%%%%%%%%%%%%%%%%%%%%%%%%%%%%%%%
\begin{abstract}
Walfisz~(1963) proved the asymptotic formula
\[
\sum_{n\le x}\phi(n)
=
\frac{3}{\pi^2}x^2+O(x(\log x)^{\frac{2}{3}}(\log\log x)^{\frac{4}{3}}),
\]
which improved the error term estimate of Mertens~(1874)
and had been the best possible estimate for more than 50 years.
Recently, H.-Q.~Liu~(2016) improved Walfisz's error term estimate to
\[
\sum_{n\le x}\phi(n)
=
\frac{3}{\pi^2}x^2+O(x(\log x)^{\frac{2}{3}}(\log\log x)^{\frac{1}{3}}).
\]
We generalize Liu's result to a certain class of arithmetic functions
and improve the result of Balakrishnan and P\'etermann~(1996).
To this end, we provide a refined version of Vinogradov's combinatorial decomposition
available for a wider class of multiplicative functions.
\end{abstract}

%%%%%%%%%%%%%%%%%%%%%%%%%%%%%%%%%%%%%%%%
\section{Introduction}
Let $\phi(n)$ be the Euler totient function
\[
\phi(n)=n\prod_{p|n}\left(1-\frac{1}{p}\right).
\]
Then it is relatively elementary to prove the asymptotic formula
\[
\sum_{n\le x}\phi(n)=\frac{3}{\pi^2}x^2+O(x\log x)
\]
due to Mertens~\cite{Mertens}.
In 1963, Walfisz~\cite{Walfisz} improved this error term estimate to
\begin{equation}
\label{Walfisz}
\sum_{n\le x}\phi(n)
=\frac{3}{\pi^2}x^2+O(x(\log x)^{\frac{2}{3}}(\log\log x)^{\frac{4}{3}}).
\end{equation}
Walfisz's improvement is based on two methods for exponential sums developed by Vinogradov,
one of which is Vinogradov's mean value theorem
and the other is some combinatorial decomposition.
Recently, H.-Q. Liu~\cite{HQ_Liu}
obtained a further improvement in the $(\log\log x)$-power:
\begin{equation}
\label{HQ_Liu}
\sum_{n\le x}\phi(n)
=\frac{3}{\pi^2}x^2+O(x(\log x)^{\frac{2}{3}}(\log\log x)^{\frac{1}{3}}).
\end{equation}
The main ingredient of Liu's improvement is to replace Vinogradov's combinatorial decompositions
by Vaughan's identity~\cite[Proposition~13.5]{Iwaniec_Kowalski},
which enables us to produce Type II sums with more efficient summation ranges.

%%%%%%%%%%%%%%%%%%%%%%%%%%%%%%%%%%%%%%%%
Walfisz's result~\cref{Walfisz}
can be generalized to a certain class of arithmetic functions.
Such a generalization was studied
by Balakrishnan and P\'etermann~\cite{Balakrishnan_Petermann}.
Their result can be summarized in the following two theorems.

%%%%%%%%%%%%%%%%%%%%%%%%%%%%%%%%%%%%%%%%
\begin{pretheorem}[{\cite[Theorem~1]{Balakrishnan_Petermann}}]
\label{Thm:BP}
Let $\alpha$ be a complex number and
\begin{equation}
\label{BP:b}
f(s)=\sum_{n=1}^{\infty}\frac{b(n)}{n^s}
\end{equation}
be a Dirichlet series absolutely convergent for $\sigma>1-\lambda$
with some real number $\lambda>0$. Define arithmetic functions
$a(n)$ and $v(n)$ by
\begin{equation}
\label{BP:a}
\sum_{n=1}^{\infty}\frac{a(n)}{n^s}
=
\zeta(s)\zeta(s+1)^{\alpha}f(s+1)
\end{equation}
and
\begin{equation}
\label{BP:v}
\sum_{n=1}^{\infty}\frac{v(n)}{n^s}=\zeta(s)^{\alpha}f(s)
\end{equation}
for $\sigma>1$, where we take the branch of $\zeta(s+1)^{\alpha}$
by $\arg\zeta(s+1)=0$ on the positive real line.
Then we have
\[
\sum_{n\le x}a(n)
=
\zeta(2)^{\alpha}f(2)x
+
\sum_{r=0}^{[\Re\alpha]}A_r(\log x)^{\alpha-r}
-\sum_{n\le y}\frac{v(n)}{n}\psi\left(\frac{x}{n}\right)
+o(1)
\]
as $x\to\infty$, where the coefficients $(A_r)$
are computable from the Laurent expansion of $\zeta(s)^{\alpha}f(s)$ at $s=1$,
$y=x\exp(-(\log x)^{\frac{1}{6}})$, $\psi(x)=\{x\}-\frac{1}{2}$
and the Landau symbol $o(1)$ depends on the assumptions of this theorem.
\end{pretheorem}

%%%%%%%%%%%%%%%%%%%%%%%%%%%%%%%%%%%%%%%%
\begin{pretheorem}[{\cite[Theorem~1]{Petermann}}]
\label{Thm:BP_error}
Let $v(n)$ be a real-valued multiplicative function. Assume that there exist
real numbers $\alpha_1,\beta\ge0$
and a sequence of real numbers $(V_r)_{r=0}^{\infty}$
such that for any positive integer $\lambda$ and real number $x\ge4$, we have
\begin{equation}
\label{H1}
\tag{\textbf{h1}}
\sum_{n\le x}|v(n)|
=
x\sum_{r=0}^{\lambda+[\alpha_1]}V_r(\log x)^{\alpha_1-r}
+
O_{\lambda}\left(x(\log x)^{-\lambda}\right),
\end{equation}
\begin{equation}
\label{H2}
\tag{\textbf{h2}}
\sum_{n\le x}|v(n)|^2\ll x(\log x)^{\beta},
\end{equation}
and
\begin{equation}
\label{H3}
\tag{\textbf{h3}}
\begin{aligned}
&\text{\hspace{2.8mm}%
the function $v(p)$ is ultimately monotonic with respect to $p$}\\
&\text{and the function $v(p^{\nu})$ is bounded for every prime $p$ and $\nu\ge1$.}
\end{aligned}
\end{equation}
Then for $x\ge4$ and $\theta>0$, we have
\[
\sum_{n\le y}\frac{v(n)}{n}\psi\left(\frac{x}{n}\right)
\ll
(\log x)^{\frac{2(\alpha_1+1)}{3}}(\log\log x)^{\frac{4(\alpha_1+1)}{3}},
\]
where $y=x\exp(-(\log x)^{\theta})$
and the implicit constant depends on $\theta$ and the constants in the above hypotheses
{\upshape\cref{H1}}, {\upshape\cref{H2}}, and {\upshape\cref{H3}}.
\end{pretheorem}

%%%%%%%%%%%%%%%%%%%%%%%%%%%%%%%%%%%%%%%%
\begin{remark}
\label{Rem:errata}
Balakrishnan and P\'etermann~\cite[Errata]{Balakrishnan_Petermann} claimed
an error in the proof of Lemma~3 of \cite[p.52--53]{Balakrishnan_Petermann},
which means, in turn, an error in the proof of Theorem~2 in the same paper.
In particular, they remarked that the argument at the top of p.53
is not immediate with the condition (h1) in \cite{Balakrishnan_Petermann}.
Therefore, we stated \cref{Thm:BP_error}, a corrected version of this theorem
given by P\'etermann~\cite{Petermann}.
However, we may use the condition (h2) of \cite{Balakrishnan_Petermann} instead of (h1)
to recover Theorem~2 of \cite{Balakrishnan_Petermann}.
We omit the details on this argument
since \cref{Thm:v_thm} below also recovers this original theorem.
For the related arguments, see the proof of \cref{Lem:TypeII} in \cref{Section:prime}.
\end{remark}

%%%%%%%%%%%%%%%%%%%%%%%%%%%%%%%%%%%%%%%%
\begin{remark}
\label{Rem:alpha}
Since the parameter $\alpha$ in \cite{Balakrishnan_Petermann}
and the parameter $\alpha$ in \cite{Petermann} have slightly different meanings,
we used the letter $\alpha_1$ to denote the parameter $\alpha$
in \cite{Petermann}. These two parameters are roughly connected
by $\alpha=\alpha_1+1$.
%To see this, apply partial summation to Lemma~2 of
%\cite{Balakrishnan_Petermann} to check the above condition \cref{H1}
%assuming $v(n)$ is positive.
\end{remark}

%%%%%%%%%%%%%%%%%%%%%%%%%%%%%%%%%%%%%%%%
\begin{remark}
In \cite{Balakrishnan_Petermann}, the parameter $y$ is given by
$y=x\exp(-(\log x)^{b})$ with some fixed real number $b$
such that $0<b<B$, where $B$ is the constant such that
$0<B\le1/2$ and we have $\zeta(s)\neq0$ in the region
\begin{equation}
\label{zero_free}
\sigma>1-(\log(|t|+4))^{-(1-B)},\quad
|t|\colon\text{large}.
\end{equation}
By using the Vinogradov-Korobov zero free region \cite[Theorem~6.1]{Ivic},
we can choose $B$ to be any positive real number $<1/3$,
where the case $B=1/3$ is excluded.
Thus we can take, for example, $b=1/6$.
We chose this specific value in \cref{Thm:BP} for the notational simplicity.
Note that the Vinogradov-Korobov zero free region \cite[Theorem~6.1]{Ivic}
is not the same one as Vinogradov and Korobov originally claimed,
which is regarded to be still unproven today.
Furthermore, by using the Selberg--Delange method as in~\cite[Chapter~II.5]{Tenenbaum},
we may take any $0<b<1/2$ just by using the above zero free region \cref{zero_free}
with $B=0$.
\end{remark}

%%%%%%%%%%%%%%%%%%%%%%%%%%%%%%%%%%%%%%%%
The main aim of this paper is to improve
the above result of Balakrishnan and P\'etermann.
As we can see from \cref{Thm:BP_error},
their result is based on Walfisz's result \cref{Walfisz}.
So it is natural to ask some improvement
up to the strength of Liu's result \cref{HQ_Liu}.
Our main result can be stated as follows.

%%%%%%%%%%%%%%%%%%%%%%%%%%%%%%%%%%%%%%%%
\begin{theorem}
\label{Thm:v_thm}
Let $v(n)$ be a complex-valued multiplicative function
such that there exists a real number $C\ge2$
satisfying the following three conditions{\upshape:}
\begin{equation}
\tag{\textbf{V1}}
\label{V1}
\text{$|v(p)|\le C$ for every prime number $p$},
\end{equation}
\begin{equation}
\tag{\textbf{V2}}
\label{V2}
\sum_{n\le x}|v(n)|^2\le Cx(\log x)^{C}\quad(x\ge4),
\end{equation}
\begin{equation}
\tag{\textbf{V3}}
\label{V3}
\sum_{p_n\le x}|v(p_{n+1})-v(p_n)|\le C(\log x)^{C}\quad(x\ge4),
\end{equation}
where $p_n$ is the $n$-th prime number.
Assume that a real number $\kappa\ge0$ satisfies
\begin{equation}
\label{V}
\tag{\textbf{V}}
\sum_{n\le x}\frac{|v(n)|}{n}\ll(\log x)^{\kappa}
\end{equation}
for $x\ge 4$.
Then for $x\ge 4$ and $\theta>0$, we have
\begin{equation}
\label{v_thm:estomate}
\sum_{n\le y}\frac{v(n)}{n}\psi\left(\frac{x}{n}\right)
\ll
(\log x)^{\frac{2\kappa}{3}}(\log\log x)^{\frac{\kappa}{3}},
\end{equation}
where $y\le x\exp(-(\log x)^{\theta})$ and
the implicit constant depends only on $\theta$, $C$,
and the implicit constant in the above condition {\upshape\cref{V}}.
\end{theorem}

%%%%%%%%%%%%%%%%%%%%%%%%%%%%%%%%%%%%%%%%
\begin{remark}
The estimate of the type \cref{V}
automatically follows by the assumption \cref{V2}. Thus, the assumption \cref{V}
is given only for specifying the exponents in the resulting estimate \cref{v_thm:estomate}.
\end{remark}

%%%%%%%%%%%%%%%%%%%%%%%%%%%%%%%%%%%%%%%%
\begin{remark}
\cref{Thm:BP}, \cref{Thm:BP_error}, and \cref{Thm:v_thm} are related to the sum
\begin{equation}
\label{Walfisz_unweighted}
\sum_{n\le x}\frac{\phi(n)}{n}
\end{equation}
rather than the left-hand side of \cref{Walfisz}.
However, it is possible to translate
this type of sums to the sum of the type \cref{Walfisz} by partial summation.
See the arguments in \cite[p.64--68]{Balakrishnan_Petermann}.
The author is also planning to give the related details somewhere.
In particular, for the sum of the type
\[
\sum_{n\le x}\phi(n)^{it}
\]
with some real number $t\neq0$, we need to employ partial summation
more carefully than in \cite[p.64--68]{Balakrishnan_Petermann}.
\end{remark}

%%%%%%%%%%%%%%%%%%%%%%%%%%%%%%%%%%%%%%%%
The estimate \cref{v_thm:estomate} gives an error term estimate
of the strength of Liu's result \cref{HQ_Liu}.
It is natural to apply Liu's approach~\cite{HQ_Liu} to prove \cref{v_thm:estomate}.
However, the author of this paper could not succeed to use Liu's approach straightforwardly.
In order to prove \cref{Thm:v_thm}, we first prepare an estimate
for the exponential sum over primes
\begin{equation}
\label{sum_over_primes}
\sum_{P<p\le P'}e\left(\frac{Q}{p}\right),\quad
P<P'\le2P
\end{equation}
similarly to Main Lemma of P\'etermann~\cite{Petermann}.
Main Lemma of P\'etermann can be improved easily
by using Liu's approach, which is just a translation of Liu's proof
for the M\"obius function in terms of the von Mangoldt function.
However, our final result deals with more general arithmetic functions $v(n)$,
and the author could not find a combinatorial identity
for such general multiplicative functions.
Thus we need to return to the original approach of Vinogradov.
In order to achieve the improved error term estimate
even by using the decomposition of the Vinogradov-style, we use
such a decomposition finer
than the decomposition used by Walfisz or by P\'etermann.
For the details, see \cref{Lem:v_sq_free}.

%%%%%%%%%%%%%%%%%%%%%%%%%%%%%%%%%%%%%%%%
Not only improving the final error term estimate itself,
\cref{Thm:v_thm} also relaxes the necessary assumptions in \cref{Thm:BP_error}.
We now compare some assumptions of \cref{Thm:BP} and \cref{Thm:v_thm}:
\begin{enumerate}
\item
We first removed the assumption that $v(n)$ is real-valued.
This is done just by remove the monotonicity on $v(p)$
as in \cref{H3} and by introducing a weaker assumption \cref{V1} and \cref{V3}.
Note that we also removed the assumption on the values of $v(n)$
at the higher prime powers.
\item The condition \cref{V2} just states the same assumption as \cref{H2}.
\item We removed the strong assumption \cref{H1} and replaced by \cref{V},
which is the same assumption as in Theorem~2 of \cite{Balakrishnan_Petermann}.
So, in particular, our \cref{Thm:v_thm} recovers
Theorem~2 of \cite{Balakrishnan_Petermann}.
\item As we mentioned in \cref{Rem:alpha}, we roughly have $\alpha_1=\alpha-1$.
In \cref{Thm:BP_error}, we assumed that $\alpha_1\ge0$
so that, in principle, we are restricted to the case $\alpha\ge1$.
In \cref{Thm:v_thm}, we roughly have $\kappa=\alpha$
and we assumed only $\kappa\ge0$. Thus, \cref{Thm:v_thm}
is applicable for a wider range of $\alpha$ than \cref{Thm:BP_error}.
\end{enumerate}
\noindent
As we can see from the above comparison,
\cref{Thm:v_thm} is available for a wider class of multiplicative function
than \cref{Thm:BP_error}.

%%%%%%%%%%%%%%%%%%%%%%%%%%%%%%%%%%%%%%%%
Actually, \cref{Thm:v_thm} is relaxed enough to obtain the following theorem.
%%%%%%%%%%%%%%%%%%%%%%%%%%%%%%%%%%%%%%%%
\begin{theorem}
\label{Thm:BP_completed}
Under the same hypothesis as in \cref{Thm:BP}, we have
\[
\sum_{n\le x}a(n)
=
\zeta(2)^{\alpha}f(2)x
+
\sum_{r=0}^{[\Re\alpha]}A_r(\log x)^{\alpha-r}
+O((\log x)^{\frac{2|\alpha|}{3}}(\log\log x)^{\frac{|\alpha|}{3}})
\]
for $x\ge4$,
where the implicit constant depends on the hypothesis in \cref{Thm:BP}.
\end{theorem}

%%%%%%%%%%%%%%%%%%%%%%%%%%%%%%%%%%%%%%%%
This enables us to apply the method of Balakrishnan--P\'etermann
directly to the generating function \cref{BP:a}
without checking any additional assumption
besides the assumptions in \cref{Thm:BP}.
We shall prove \cref{Thm:BP_completed} in \cref{Section:BP_completed}.

%%%%%%%%%%%%%%%%%%%%%%%%%%%%%%%%%%%%%%%%
\begin{remark}
Recently, Drappeau and Topacogullari~\cite{Drappeau_Topacogullari}
gave a new combinatorial decomposition for general $\tau_{\alpha}(n)$ with any complex number $\alpha$
in the context of the generalized Titchmarsh divisor problems.
As we will see in the proof of \cref{Thm:BP_completed},
especially in the argument after \cref{BP_completed:decompose},
the function $v(n)$ in \cref{Thm:BP} can be reduced to $\tau_{\alpha}(n)$.
More precisely, $v(n)$ can be decomposed into the convolution of $\tau_{\alpha}(n)$
and an arithmetic function $b(n)$ for which the series \cref{BP:b} converges absolutely for $\sigma<1$.
(For the same principle, see also Lemma~2.2 of \cite{Drappeau_Topacogullari}.)
Thus, for the proof of \cref{Thm:BP_completed},
the method of Drappeau and Topacogullari is available.
However, our \cref{Thm:v_thm} deals with a slightly wider class of arithmetic functions.
For example, the multiplicative function $v(n)$ defined by
\[
v(p)=2+\frac{1}{\log\log(p+4)},\quad
v(p^{\nu})=0\quad(\nu\ge 2)
\]
satisfies the conditions \cref{V1}, \cref{V2}, and \cref{V3},
but cannot be decomposed into the convolution of
$\tau_{\alpha}(n)$ and $b(n)$ for which the series \cref{BP:b} converges for $\sigma<1$.
Indeed, the only possible choice of $\alpha$ for this example is $\alpha=2$,
but for this case,
\[
b(p)=v(p)+\tau_{-2}(p)=\frac{1}{\log\log(p+4)}
\]
so that
\[
\sum_{n\le x}\frac{|b(n)|}{n}
\ge
\sum_{p\le x}\frac{|b(p)|}{p}
=
\sum_{p\le x}\frac{1}{p\log\log(p+4)}
\to
\infty
\]
as $x\to\infty$. Thus, at least in \cref{Thm:v_thm},
our combinatorial decomposition seems to be slightly more general than
the decomposition of Drappeau and Topacogullari.
\end{remark}

%%%%%%%%%%%%%%%%%%%%%%%%%%%%%%%%%%%%%%%%
\section{Notation}
The letter $p$ denotes a prime number
and $p_n$ denotes the $n$-th prime number.
By $s=\sigma+it$, we denote a complex variable $s$.

For a positive integer $n$,
we denote by $\omega(n)$ the number of distinct prime factors of $n$
and by $\Omega(n)$ the number of prime factors of $n$ counted with multiplicity.
As usual, $\Lambda(n)$ is the von Mangoldt function,
$\mu(n)$ is the M\"obius function, $\phi(n)$ is the Euler totient function,
and $\sigma(n)$ is the divisor summatory function.
The function $\tau(n)$ is the divisor function,
i.e.~it denotes the number of positive divisors of $n$.
More generally, for a complex number $\alpha$,
we define the divisor function $\tau_{\alpha}(n)$ by the generation function
\[
\zeta(s)^{\alpha}=\sum_{n=1}^{\infty}\frac{\tau_{\alpha}(n)}{n^s}\quad\sigma>1,
\]
where the branch of $\zeta(s)^{\alpha}$ is taken by $\arg\zeta(s)=0$ for $s>1$.
Note that $\tau_{2}(n)=\tau(n)$.

For a positive integer $n$, we define $p_{\max}(n)$ and $p_{\min}(n)$
be the largest and the smallest prime factor of $n>1$, respectively,
and as a convention, we define $p_{\max}(1)=1$ and $p_{\min}(1)=+\infty$.
By $\psi(x,y)$, we denote the number of $y$-smooth numbers $\le x$, i.e.
\[
\psi(x,y)=|\{n\le x\mid p_{\max}(n)\le y\}|.
\]

We use the conditions
\begin{center}
\cref{H1},\ \cref{H2},\ \cref{H3},\ 
\cref{V1},\ \cref{V2},\ \cref{V3},\ \cref{V}
\end{center}
on multiplicative functions. See \cref{Thm:BP} and \cref{Thm:v_thm}.
The letters $C$ and $\kappa$ always denote the constants
in \cref{V1},\ \cref{V2},\ \cref{V3},\ \cref{V}.
By saying a multiplicative function,
we exclude the constant function $0$.

We denote the fractional part of a real number $x$ by $\{x\}$
and let
\[
\psi(x)=\{x\}-\frac{1}{2}.
\]
The function $e(x)$ is defined by $e(x)=e^{2\pi ix}$ as usual.

The letter $B$ denotes the constant used for describing admissible ranges
of several parameters in each Theorem or Lemma,
e.g.~see the condition \cref{TypeII:UV_range}.
Thus $B$ has the same meaning
during a fixed Theorem or Lemma and their proof,
but it may have different meanings in different context.
In order to denote the constant $B$ used in some preceding context,
we use letters $B_1,B_2,\ldots$ instead of $B$.
We emphasize the dependance of $B$ on some letters $A,C,\ldots$
by writing $B=B(A,C,\ldots)$.

If Theorem or Lemma is stated
with the phrase ``where the implicit constant depends only on $a,b,c,\ldots$'',
then every implicit constant in the corresponding proof
may also depend on $a,b,c,\ldots$ even without special mentions.

%%%%%%%%%%%%%%%%%%%%%%%%%%%%%%%%%%%%%%%%
\section{Exponential sums over primes}
\label{Section:prime}
In this section, we prepare an estimate for exponential sums over primes.
As we have mentioned,
this estimate corresponds to Main Lemma of P\'etermann~\cite{Petermann}
or the result of Liu~\cite{HQ_Liu}. We follow the method of Liu in order to
improve Main Lemma of P\'etermann~\cite{Petermann} and simplify the proof.

We first prepare an estimate for the exponential sum
\[
\sum_{P<n\le P'}e\left(\frac{Q}{n}\right),\quad
P<P'\le2P,\quad
e(x)=e^{2\pi ix}
\]
based on Vinogradov's mean value theorem.
We use Vinogradov's mean value theorem through an application
of the following lemma due to Karatsuba~\cite{Karatsuba}.

%%%%%%%%%%%%%%%%%%%%%%%%%%%%%%%%%%%%%%%%
\begin{lemma}[Karatsuba's lemma~{\cite[Theorem 1]{Karatsuba}}]
\label{Lem:Karatsuba}
Let $k$ be a positive integer, $X,P$ be real numbers with $P\ge1$
and $f(x)$ be a real-valued function defined 
and $(k+1)$-times continuously differentiable on $[X,X+P]$.
Assume that there are four constants
\begin{equation}
\label{Karatsuba:c}
0<c_0<1,\quad
0<c_3\le c_2<c_1<1
\end{equation}
and positive integers
\begin{equation}
\label{Karatsuba:rj}
c_0k\le r\le k,\quad
1\le j_1<j_2<\cdots<j_r\le k
\end{equation}
satisfying the following conditions{\upshape:}
\begin{enumerate}
\renewcommand{\labelenumi}{(\Alph{enumi})}
\item we have
\[
\left|\frac{f^{(k+1)}(x)}{(k+1)!}\right|\le P^{-c_1(k+1)}
\]
on $[X,X+P]$,
\item we have for every $j\in\{j_1,\ldots,j_r\}$,
\[
P^{-c_2j}\le\left|\frac{f^{(j)}(x)}{j!}\right|\le P^{-c_3j}
\]
on $[X,X+P]$.
\end{enumerate}
Then there is a constant $0<\gamma\le1$ such that for any $P_1\le P$,
\[
\left|\sum_{X<n\le X+P_1}e(f(n))\right|\ll P^{1-\frac{\gamma}{k^2}},
\]
where the constant $\gamma$ and the implicit constant
depends only on $c_0,c_1,c_2$, and $c_3$.
\end{lemma}

%%%%%%%%%%%%%%%%%%%%%%%%%%%%%%%%%%%%%%%%
\begin{remark}
In \cite{Karatsuba}, there is one more constant $c_4$.
However, since this constant $c_4$ can be taken by $c_4=(c_1-c_2)/2$
in the above setting, we did not introduce the constant $c_4$
for the notational simplicity.
\end{remark}

%%%%%%%%%%%%%%%%%%%%%%%%%%%%%%%%%%%%%%%%
We also make use of the Kusmin--Landau inequality:
\begin{lemma}[Kusmin--Landau inequality~{\cite[Theorem~2.1]{Graham_Kolesnik}}]
\label{Lem:Kusmin_Landau}
Let $X,P,\lambda$ be real numbers with $P\ge1$ and $\lambda>0$,
and $f(x)$ be a real-valued function defined and continuously differentiable on $[X,X+P]$.
Assume also that $f'(x)$ is monotonic and $\lambda<|f'(x)|<1-\lambda$ on $[X,X+P]$.
Then, we have
\[
\sum_{X<n\le X+P}e(f(n))\ll\lambda^{-1},
\]
where the implicit constant is absolute.
\end{lemma}

%%%%%%%%%%%%%%%%%%%%%%%%%%%%%%%%%%%%%%%%
\begin{lemma}
\label{Lem:Walfisz}
Let $P,P',Q\ge4$ be real numbers with $P<P'\le 2P$.
Then
\[
\sum_{P<n\le P'}e\left(\frac{Q}{n}\right)
\ll
P\exp\left(-\gamma\frac{(\log P)^3}{(\log Q)^2}\right)+P^2Q^{-1},
\]
where the implicit constant and the constant $\gamma>0$ are absolute.
\end{lemma}
%%%%%%%%%%%%%%%%%%%%%%%%%%%%%%%%%%%%%%%%
\begin{proof}
Let $f(x)=Q/x$.
We may assume $P>2^{12}$ since otherwise
\[
P\exp\left(-\gamma\frac{(\log P)^3}{(\log Q)^2}\right)
\ge
P\exp\left(-\frac{(12\log 2)^3}{(2\log 2)^2}\right)
\gg
P
\]
so the assertion is trivial.
We may also assume $P\le Q^{\frac{2}{3}}/2$
since if $P>Q^{\frac{2}{3}}/2$, then we may apply \cref{Lem:Kusmin_Landau}.
Indeed, in this case, $f'(x)$ is increasing on $[P,2P]$,
\[
|f'(x)|\le QP^{-2}\le2^{\frac{3}{2}}P^{-\frac{1}{2}}\le\frac{1}{2}
\]
and $|f'(x)|\gg QP^{-2}$ on $[P,2P]$.
Thus, \cref{Lem:Kusmin_Landau} gives
\[
\sum_{P<n\le P'}e\left(\frac{Q}{n}\right)\ll P^2Q^{-1}.
\]
Therefore, we shall consider the case $2^{12}<P\le Q^{\frac{2}{3}}/2$.

We apply \cref{Lem:Karatsuba} with
\begin{equation}
\label{EQ:Pkc_choice}
\begin{gathered}
X=P,\quad
P_1=P'-P,\quad
k=\left[26\frac{\log Q}{\log P}\right],\\
c_0=\frac{1}{39},\quad
c_1=\frac{25}{26},\quad
c_2=\frac{23}{24},\quad
c_3=\frac{1}{4},
\end{gathered}
\end{equation}
and
\begin{equation}
\label{EQ:j_choice}
j_1<\cdots<j_r\colon
\text{all integers in
$\left(2\frac{\log Q^{\frac{2}{3}}}{\log P},4\frac{\log Q^{\frac{2}{3}}}{\log P}\right]$}
\end{equation}
so that
\[
r
=\left[4\frac{\log Q^{\frac{2}{3}}}{\log P}\right]-\left[2\frac{\log Q^{\frac{2}{3}}}{\log P}\right].
\]
With this choice, the condition \cref{Karatsuba:c}
is obviously satisfied. For the condition \cref{Karatsuba:rj}, the size of $j_1,\ldots,j_r$
clearly satisfies the condition. Also, we have
\[
r
=\left[4\frac{\log Q^{\frac{2}{3}}}{\log P}\right]-\left[2\frac{\log Q^{\frac{2}{3}}}{\log P}\right]
\ge2\frac{\log Q^{\frac{2}{3}}}{\log P}-1\ge\frac{\log Q^{\frac{2}{3}}}{\log P}\ge1
\]
so that $j_1,\ldots,j_r$ is not an empty sequence and
\[
c_0k\le\frac{2}{3}
\frac{\log Q}{\log P}\le r\le 4\frac{\log Q^{\frac{2}{3}}}{\log P}\le k.
\]
This assures the condition \cref{Karatsuba:rj}.
The remaining conditions for \cref{Lem:Karatsuba} are (A) and (B).
The condition (A) can be checked as
\[
\left|\frac{f^{(k+1)}(x)}{(k+1)!}\right|
=\frac{Q}{x^{k+2}}
\le
\frac{Q}{P^{k+2}}
\le
\frac{Q}{P^{k+1}}
=
P^{-c_1(k+1)}
\frac{Q}{P^{(1-c_1)(k+1)}}
\le
P^{-c_1(k+1)},
\]
where we used
\[
(1-c_1)(k+1)
=\frac{1}{26}\left(\left[26\frac{\log Q}{\log P}\right]+1\right)
\ge\frac{\log Q}{\log P}.
\]
We move on to the condition (B).
For $j\in\{j_1,\ldots,j_r\}$, by definition, we have
\begin{equation}
\label{Walfisz:j_range}
2\frac{\log Q^{\frac{2}{3}}}{\log P}<j\le4\frac{\log Q^{\frac{2}{3}}}{\log P}.
\end{equation}
Therefore, for the upper bound of the condition (B), we find
\[
\left|\frac{f^{(j)}(x)}{j!}\right|
=\frac{Q}{x^{j+1}}
\le\frac{Q}{P^{j}}
\le P^{-c_3j}.
\]
For the lower bound, by using the assumption $P\le Q^{\frac{2}{3}}/2$, we obtain
\[
\left|\frac{f^{(j)}(x)}{j!}\right|
\ge\frac{Q}{(2P)^{j+1}}
\ge\frac{Q^{\frac{1}{3}}}{(2P)^j}.
\]
Since we are also assuming $2^{12}<P$,
\[
\left|\frac{f^{(j)}(x)}{j!}\right|
\ge
\frac{Q^{\frac{1}{3}}}{P^{\frac{13j}{12}}}
=
P^{-c_2j}\frac{Q^{\frac{1}{3}}}{P^{\frac{j}{8}}}
\ge
P^{-c_2j}
\]
by using the inequality \cref{Walfisz:j_range}.

Therefore, we can apply \cref{Lem:Karatsuba} to obtain
\[
\sum_{P<n\le P'}e\left(\frac{Q}{n}\right)
\ll
P^{1-\frac{\gamma}{k^2}}
\]
where the implicit constant and $\gamma$ is now absolute.
Since
\[
k=\left[26\frac{\log Q}{\log P}\right]\le26\frac{\log Q}{\log P},
\]
we obtain
\[
P^{1-\frac{\gamma}{k^2}}
\le
P\exp\left(-\frac{\gamma}{676}\frac{(\log P)^3}{(\log Q)^2}\right).
\]
This completes the proof.
\end{proof}

%%%%%%%%%%%%%%%%%%%%%%%%%%%%%%%%%%%%%%%%
We next recall Vaughan's identity,
the main ingredient of Liu's improvement.
\begin{lemma}[Vaughan's identity]
\label{Lem:Vaughan}
For a real number $z\ge2$ and any integer $n>z$,
\[
\Lambda(n)=a_1(n)-a_2(n)+a_3(n),
\]
where
\[
a_1(n)
=\sum_{\substack{uv=n\\u\le z}}\mu(u)\log v,\quad
a_2(n)
=
\sum_{\substack{uv=n\\u\le z^2}}
\left(\sum_{\substack{dm=u\\d,m\le z}}\mu(d)\Lambda(m)\right),
\]
\[
a_3(n)
=
\sum_{\substack{uv=n\\u,v>z}}
\left(\sum_{\substack{dm=u\\d>z}}\mu(d)\right)\Lambda(v).
\]
\end{lemma}
%%%%%%%%%%%%%%%%%%%%%%%%%%%%%%%%%%%%%%%%
\begin{proof}
See \cite[p. 344, Section 13.4]{Iwaniec_Kowalski}.
\end{proof}

%%%%%%%%%%%%%%%%%%%%%%%%%%%%%%%%%%%%%%%%
After applying Vaughan's identity, the exponential sum is decomposed
into Type~I and Type~II sums as usual. 
Therefore, we next prove the Type~II sum estimate
with exponential sum estimate of the Vinogradov type.

%%%%%%%%%%%%%%%%%%%%%%%%%%%%%%%%%%%%%%%%
\begin{lemma}[Type II sum estimate]
\label{Lem:TypeII}
For any real number $A\ge1$, there exists a real number $B=B(A)\ge1$
such that for any sequences of complex numbers
\[
\mathcal{A}=(\alpha_u)_{u=1}^{\infty},\quad
\mathcal{B}=(\beta_v)_{v=1}^{\infty}
\]
and any real numbers $P,P',U,U',V,V',Q\ge4$ with
\[
P<P'\le 2P,\quad
U<U'\le 2U,\quad
V<V'\le 2V',
\]
we have
\[
\sum_{\substack{P<uv\le P'\\U<u\le U'\\V<v\le V'}}
\alpha_u\beta_ve\left(\frac{Q}{uv}\right)
\ll
\left(P^{\frac{1}{2}}(\log Q)^{-A}+PQ^{-\frac{1}{2}}\right)
\|\mathcal{A}\|\|\mathcal{B}\|(\log Q)^{\frac{1}{2}}
\]
provided
\begin{equation}
\label{TypeII:UV_range}
U,V\ge\exp(B(\log Q)^{\frac{2}{3}}(\log\log Q)^{\frac{1}{3}}),
\end{equation}
where
\[
\|\mathcal{A}\|^2=\sum_{U<u\le U'}|\alpha_u|^2,\quad
\|\mathcal{B}\|^2=\sum_{V<v\le V'}|\beta_v|^2
\]
and the implicit constant depends only on $A$.
\end{lemma}
%%%%%%%%%%%%%%%%%%%%%%%%%%%%%%%%%%%%%%%%
\begin{proof}
We let $0<\gamma\le1$ be the constant in \cref{Lem:Walfisz}
and take $B=(2A/\gamma)^{\frac{1}{3}}$.
We then employ some trivial reductions.
By symmetry between $u$ and $v$,
it suffices to consider the case $U\ge V$.
If the sum in the assertion is non-empty,
then we have $P<U'V'$ and $UV\le P'$. Therefore, we may assume
\begin{equation}
\label{TypeII:UV_size}
P/4\le UV\le 2P.
\end{equation}
We may assume $Q$ is larger than some constant depending only on $A$
since otherwise the Cauchy--Schwarz inequality gives
\begin{equation}
\label{TypeII:trivial}
\sum_{\substack{P<uv\le P'\\U<u\le U'\\V<v\le V'}}
\alpha_u\beta_ve\left(\frac{Q}{uv}\right)
\ll
\|\mathcal{A}\|\|\mathcal{B}\|(UV)^{\frac{1}{2}}
\ll
\|\mathcal{A}\|\|\mathcal{B}\|P^{\frac{1}{2}}
\end{equation}
so that the assertion follows.
Similarly, we may assume $P\le Q/32$ since otherwise
\[
PQ^{-\frac{1}{2}}\gg P^{\frac{1}{2}}
\]
so that \cref{TypeII:trivial} again gives the assertion.

After the above reductions, we use the Cauchy--Schwarz inequality to obtain
\begin{equation}
\label{TypeII:CS}
\Bigg|\sum_{\substack{P<uv\le P'\\U<u\le U'\\V<v\le V'}}
\alpha_u\beta_ve\left(\frac{Q}{uv}\right)\Bigg|^2
\le
\|\mathcal{A}\|^2
\sum_{U<u\le U'}\Bigg|\sum_{\substack{P/u<v\le P'/u\\V<v\le V'}}\beta_v
e\left(\frac{Q}{uv}\right)\Bigg|^2.
\end{equation}
We then expand the square in the latter factor as
\begin{equation}
\label{TypeII:open_sq}
\begin{aligned}
&\sum_{U<u\le U'}\Bigg|\sum_{\substack{P/u<v\le P'/u\\V<v\le V'}}\beta_v
e\left(\frac{Q}{uv}\right)\Bigg|^2\\
&\le
\sum_{V<v_1,v_2\le V'}|\beta_{v_1}\beta_{v_2}|
\left|\sum_{U(v_1,v_2)< u\le U'(v_1,v_2)}
e\left(\frac{Q}{u}\left(\frac{1}{v_1}-\frac{1}{v_2}\right)\right)\right|\\
&=
\sum_{v_1=v_2}+2\sum_{v_1<v_2}
=
S_{1}+2S_{2},\quad\text{say},
\end{aligned}
\end{equation}
where
\[
U(v_1,v_2)=\max(U,P/v_1,P/v_2),\quad
U'(v_1,v_2)=\min(U',P'/v_1,P'/v_2).
\]
Note that obviously
\[
U\le U(v_1,v_2),\quad
U'(v_1,v_2)\le U'\le 2U\le 2U(v_1,v_2).
\]
for every $v_1,v_2$. Thus for $S_{1}$, we have
\begin{equation}
\label{TypeII:diagonal}
S_{1}
\ll U\|\mathcal{B}\|^2
\ll PV^{-1}\|\mathcal{B}\|^2
\ll P(\log Q)^{-2A}\|\mathcal{B}\|^2
\end{equation}
by \cref{TypeII:UV_range} and \cref{TypeII:UV_size}.
We apply \cref{Lem:Walfisz} for $S_2$.
If $V<v_1<v_2\le V'$, then
\[
Q\left(\frac{1}{v_1}-\frac{1}{v_2}\right)
=
\frac{Q(v_2-v_1)}{v_1v_2}
\ge \frac{Q}{4V^2}
\ge \frac{Q}{4UV}
\ge \frac{Q}{8P}
\ge 4
\]
since $U\ge V$, $UV\le P'$ and $P\le Q/32$.
Thus we may apply \cref{Lem:Walfisz} to the sum
\[
\sum_{U(v_1,v_2)<u\le U'(v_1,v_2)}
e\left(\frac{Q}{u}\left(\frac{1}{v_1}-\frac{1}{v_2}\right)\right).
\]
By using
\[
U\le U(v_1,v_2)\le 2PV^{-1},\quad
Q\left(\frac{1}{v_1}-\frac{1}{v_2}\right)\le Q
\]
and
\[
U(v_1,v_2)^2\left(Q\left(\frac{1}{v_1}-\frac{1}{v_2}\right)\right)^{-1}
\ll
(PV^{-1})^2Q^{-1}\frac{v_1v_2}{|v_1-v_2|}
\ll
\frac{P^2Q^{-1}}{|v_1-v_2|}
\]
we obtain
\begin{align}
\sum_{U(v_1,v_2)<u\le U'(v_1,v_2)}
e\left(\frac{Q}{u}\left(\frac{1}{v_1}-\frac{1}{v_2}\right)\right)
&\ll
PV^{-1}\exp\left(-\gamma\frac{(\log U)^3}{(\log Q)^2}\right)
+
\frac{P^2Q^{-1}}{|v_1-v_2|}\\
&\ll
PV^{-1}(\log Q)^{-2A}+\frac{P^2Q^{-1}}{|v_1-v_2|}
\end{align}
by \cref{TypeII:UV_range}
since we took the constant $B$ by $B=(2A/\gamma)^{\frac{1}{3}}$.
By using
\begin{align}
\sum_{V<v_1,v_2\le V'}|\beta_{v_1}\beta_{v_2}|
&\le
V\sum_{V<v\le V'}|\beta_{v}|^2
=
V\|\mathcal{B}\|^2
\end{align}
and
\[
\sum_{V<v_1\neq v_2\le V'}\frac{|\beta_{v_1}\beta_{v_2}|}{|v_1-v_2|}
\ll
\sum_{V<v_1\le V'}|\beta_{v_1}|^2
\sum_{\substack{V<v_2\le V'\\v_1\neq v_2}}\frac{1}{|v_1-v_2|}
\ll
(\log Q)\|\mathcal{B}\|^2,
\]
where we used the symmetry between $v_1$ and $v_2$,
we arrive at
\begin{equation}
\label{TypeII:non_diagonal}
\begin{aligned}
S_2
&\ll
\left(P(\log Q)^{-2A}
+
P^2Q^{-1}(\log Q)\right)\|\mathcal{B}\|^2\\
&\ll
\left(P(\log Q)^{-2A}
+
P^2Q^{-1}\right)\|\mathcal{B}\|^2(\log Q).
\end{aligned}
\end{equation}
By combining \cref{TypeII:diagonal} and \cref{TypeII:non_diagonal}
and inserting into \cref{TypeII:open_sq},
\[
\sum_{U<u\le U'}\Bigg|\sum_{\substack{P/u<v\le P/u'\\V<v\le V'}}\beta_v
e\left(\frac{Q}{uv}\right)\Bigg|^2
\ll
\left(P(\log Q)^{-2A}+P^2Q^{-1}\right)\|\mathcal{B}\|^2(\log Q).
\]
Substituting this estimate into \cref{TypeII:CS},
we obtain the lemma.
\end{proof}

%%%%%%%%%%%%%%%%%%%%%%%%%%%%%%%%%%%%%%%%
We now consider the exponential sum over primes.

%%%%%%%%%%%%%%%%%%%%%%%%%%%%%%%%%%%%%%%%
\begin{lemma}
\label{Lem:Liu}
For any real number $A\ge1$, there exists a real number $B=B(A)\ge1$
such that for any real numbers $P,P',Q\ge4$ with $P<P'\le 2P$,
\[
\sum_{P<n\le P'}\Lambda(n)e\left(\frac{Q}{n}\right)
\ll
\left(P(\log Q)^{-A}+P^{\frac{3}{2}}Q^{-\frac{1}{2}}\right)(\log Q)^5
\]
provided
\begin{equation}
\label{Liu:UV_range} 
P\ge\exp(B(\log Q)^{\frac{2}{3}}(\log\log Q)^{\frac{1}{3}}),
\end{equation}
where the implicit constant depends only on $A$.
\end{lemma}
%%%%%%%%%%%%%%%%%%%%%%%%%%%%%%%%%%%%%%%%
\begin{proof}
We let $0<\gamma\le1$ be the constant in \cref{Lem:Walfisz}
and $B_1(A)$ be the constant $B(A)$ in \cref{Lem:TypeII}
and take $B=\max(3(A/\gamma)^{\frac{1}{3}},3B_1(A))$ for the current proof.
We may assume that $Q$ is larger than some constant
depending only on $A$ since otherwise the trivial estimate
\[
\sum_{P<n\le P'}\Lambda(n)e\left(\frac{Q}{n}\right)
\ll
P
\]
is enough.
Also, if $P>Q$, then we have $P^{\frac{3}{2}}Q^{-\frac{1}{2}}\gg P$
so the assertion again trivially follows. Hence we may further assume $P\le Q$.

We use \cref{Lem:Vaughan} with $z=P^{\frac{1}{3}}$.
This gives a decomposition
\begin{equation}
\label{Liu:Vaughan}
\sum_{P<n\le P'}\Lambda(n)e\left(\frac{Q}{n}\right)
=
S_1-S_2+S_3,
\end{equation}
where
\[
S_1=
\sum_{\substack{P<uv\le P'\\u\le P^{\frac{1}{3}}}}
\mu(u)(\log v)e\left(\frac{Q}{uv}\right),\quad
S_2=
\sum_{\substack{P<uv\le P'\\u\le P^{\frac{2}{3}}}}c_2(u)e\left(\frac{Q}{uv}\right),
\]
\[
S_3=
\sum_{\substack{P<uv\le P'\\u,v>P^{\frac{1}{3}}}}
c_3(u)\Lambda(v)e\left(\frac{Q}{uv}\right),
\]
where the coefficients $c_2(u)$ and $c_3(u)$ are given by
\[
c_2(u)=\sum_{\substack{dm=u\\d,m\le P^{\frac{1}{3}}}}\mu(d)\Lambda(m),\quad
c_3(u)=\sum_{\substack{dm=u\\d>P^{\frac{1}{3}}}}\mu(d).
\]
Note that
\begin{equation}
\label{Liu:coefficient}
|c_2(u)|\le\sum_{dm=u}\Lambda(m)=\log u,\quad
|c_3(u)|\le\tau(u).
\end{equation}

For the sum $S_1$, we start with
\begin{equation}
\label{Liu:S1}
S_1
=
\sum_{u\le P^{\frac{1}{3}}}\mu(u)\sum_{P/u<v\le P'/u}(\log v)e\left(\frac{Q/u}{v}\right).
\end{equation}
We apply \cref{Lem:Walfisz} to the inner sum.
Note that
\[
4\le Q^{\frac{2}{3}}\le Q/u\le Q,\quad
P/u\ge P^{\frac{2}{3}}\ge P^{\frac{1}{3}}
\]
for this inner sum. Thus, for $P/u<x\le 2P/u$, we obtain
\begin{equation}
\label{Liu:TypeI_inner}
\begin{aligned}
\sum_{P/u<v\le x}e\left(\frac{Q/u}{v}\right)
&\ll
\left(P\exp\left(-\frac{\gamma}{27}\frac{(\log P)^3}{(\log Q)^2}\right)
+P^2Q^{-1}\right)u^{-1}\\
&\ll
\left(P(\log Q)^{-A}+P^2Q^{-1}\right)u^{-1}
\end{aligned}
\end{equation}
by \cref{Liu:UV_range}  since $B\ge3(A/\gamma)^{\frac{1}{3}}$.
By partial summation,
\begin{align}
\sum_{P/u<v\le P'/u}(\log v)e\left(\frac{Q/u}{v}\right)
&\ll
(\log Q)\sup_{P/u<x\le P'/u}\left|\sum_{P/u<v\le x}e\left(\frac{Q/u}{v}\right)\right|\\
&\ll
\left(P(\log Q)^{-A}+P^2Q^{-1}\right)u^{-1}(\log Q)
\end{align}
We substitute this estimate into \eqref{Liu:S1}.
Then by using $PQ^{-1}\le P^{\frac{1}{2}}Q^{-\frac{1}{2}}$ and
\[
\sum_{u\le P^{\frac{1}{3}}}\frac{1}{u}\ll\log P\ll\log Q,
\]
we arrive at the desired estimate
\begin{equation}
\label{Liu:S1_estimate}
S_1
\ll
\left(P(\log Q)^{-A}+P^{\frac{3}{2}}Q^{-\frac{1}{2}}\right)(\log Q)^{2}
\end{equation}
for $S_1$.

The sum $S_2$ is treated similarly to $S_1$. We start with
\begin{equation}
\label{EQ:Type_I_Liu_other}
S_2
=
\sum_{u\le P^{\frac{2}{3}}}c_2(u)\sum_{P/u<v\le P'/u}e\left(\frac{Q/u}{v}\right).
\end{equation}
For the inner sum, we again have
\[
4\le Q/u\le Q,\quad
P/u\ge P^{\frac{1}{3}}
\]
so that the estimate \cref{Liu:TypeI_inner} is available.
Hence by using \cref{Liu:coefficient} and
\[
\sum_{u\le P^{\frac{2}{3}}}\frac{\log u}{u}\ll(\log P)^2\ll(\log Q)^2,
\]
we arrive at the desired estimate
\begin{equation}
\label{Liu:S2_estimate}
S_2
\ll
\left(P(\log Q)^{-A}+P^{\frac{3}{2}}Q^{-\frac{1}{2}}\right)(\log Q)^2
\end{equation}
for $S_2$.

For the sum $S_3$, we can employ dyadic subdivision to obtain
\begin{equation}
\label{Liu:S3_dissect}
S_3\ll (\log Q)^2\sup|S_3(U,U',V,V')|,
\end{equation}
where the supremum is taken over real numbers $U,U',V,V'$
with the conditions
\[
U<U'\le2U,\quad
V<V'\le2V,\quad
U,V\ge P^{\frac{1}{3}},\quad
UV\le P'
\]
and $S_3(U,U',V,V')$ is defined by
\begin{equation}
S_3(U,U',V,V')
=
\sum_{\substack{P<uv\le P'\\U<u\le U'\\V<v\le V'}}
c_3(u)\Lambda(v)e\left(\frac{Q}{uv}\right).
\end{equation}
We apply \cref{Lem:TypeII} to this double sum.
Since our choice of $B$ gives $B\ge3B_1(A)$,
by \cref{Liu:UV_range}, we find that
\[
U,V\ge P^{\frac{1}{3}}\ge\exp(B_1(A)(\log Q)^{\frac{2}{3}}(\log\log Q)^{\frac{1}{3}}).
\]
Also, we have
\begin{gather}
\sum_{U<u\le U'}|c_3(u)|^2
\le\sum_{U<u\le U'}\tau(u)^2
\ll U(\log U)^3,\\
\sum_{V<v\le V'}\Lambda(v)^2
\le(\log V)\sum_{V<v\le V'}\Lambda(v)
\ll V\log V.
\end{gather}
Thus \cref{Lem:TypeII} and $UV\le P'$ imply
\begin{align}
S_3(U,U',V,V')
&\ll
\left(P^{\frac{1}{2}}(\log Q)^{-A}+PQ^{-\frac{1}{2}}\right)
(UV)^{\frac{1}{2}}(\log Q)^{\frac{5}{2}}\\
&\ll
\left(P(\log Q)^{-A}+P^{\frac{3}{2}}Q^{-\frac{1}{2}}\right)(\log Q)^{3}.
\end{align}
By returning to \eqref{Liu:S3_dissect}, we find that
\[
S_3\ll
\left(P(\log Q)^{-A}+P^{\frac{3}{2}}Q^{-\frac{1}{2}}\right)(\log Q)^5.
\]
Combining this with \cref{Liu:Vaughan}, \cref{Liu:S1_estimate}
and \cref{Liu:S2_estimate} we arrive at the desired estimate.
\end{proof}

%%%%%%%%%%%%%%%%%%%%%%%%%%%%%%%%%%%%%%%%
\begin{lemma}
\label{Lem:prime}
For any real number $A\ge1$, there exists a real number $B=B(A)\ge1$
such that for any real numbers $P,P',Q\ge4$ with $P<P'\le 2P$,
\[
\sum_{P<p\le P'}e\left(\frac{Q}{p}\right)
\ll
\left(P(\log Q)^{-A}+P^{\frac{3}{2}}Q^{-\frac{1}{2}}\right)(\log Q)^{5}
\]
provided
\begin{equation}
\label{prime:UV_range}
P\ge\exp(B(\log Q)^{\frac{2}{3}}(\log\log Q)^{\frac{1}{3}}),
\end{equation}
where the implicit constant depends only on $A$.
\end{lemma}
%%%%%%%%%%%%%%%%%%%%%%%%%%%%%%%%%%%%%%%%
\begin{proof}
We may assume $Q$ is larger than some absolute constant depending only on $A$.
Let us take the same constant $B=B(A)\ge1$ as in \cref{Lem:Liu}.
We replace $\Lambda(n)$ by $\log p$.
This replacement produces an error
\begin{equation}
\label{prime:prime_power_error}
\le
\sum_{2\le\nu\le2\log P}\sum_{p\le (2P)^{\frac{1}{\nu}}}\log p
\ll
P^{\frac{1}{2}}(\log P)^2.
\end{equation}
By \cref{prime:UV_range} and $B\ge1$, we find that
\[
\log P\ge B(\log Q)^{\frac{2}{3}}(\log\log Q)^{\frac{1}{3}}\ge (\log Q)^{\frac{2}{3}}
\]
so that the error \cref{prime:prime_power_error} is bounded as
\[
P(\log Q)^{-A}\ge P(\log P)^{-\frac{3A}{2}}\gg P^{\frac{1}{2}}(\log P).
\]
Therefore, \cref{Lem:Liu} implies
\[
\sum_{P<p\le P'}(\log p)e\left(\frac{Q}{p}\right)
\ll
\left(P(\log Q)^{-A}+P^{\frac{3}{2}}Q^{-\frac{1}{2}}\right)(\log Q)^{5}
\]
provided \cref{prime:UV_range}.
By using partial summation, we obtain the lemma.
\end{proof}

%%%%%%%%%%%%%%%%%%%%%%%%%%%%%%%%%%%%%%%%
\section{Exponential sums with multiplicative functions}
\label{Section:v}
Our next task is to estimate the exponential sum
\[
\sum_{P<n\le P'}v(n)e\left(\frac{Q}{n}\right),\quad
P<P'\le2P
\]
by using the result of \cref{Section:prime},
where $v(n)$ is a multiplicative function satisfying the conditions
\cref{V1}, \cref{V2}, and \cref{V3}.

%%%%%%%%%%%%%%%%%%%%%%%%%%%%%%%%%%%%%%%%
We first prepare an estimate for the exponential sum over primes
with the coefficient $v(p)$. This is the point where the condition
\cref{V3} is used, which states $v(p)$ has a bounded variation.

%%%%%%%%%%%%%%%%%%%%%%%%%%%%%%%%%%%%%%%%
\begin{lemma}
\label{Lem:v_prime}
Let $v(n)$ be a complex-valued multiplicative function
satisfying {\upshape\cref{V1}} and {\upshape\cref{V3}}.
For any real number $A\ge1$, there exists a real number $B=B(A)\ge1$
such that for any real numbers $P,P',Q\ge4$ with $P<P'\le 2P$,
\[
\sum_{P<p\le P'}v(p)e\left(\frac{Q}{p}\right)
\ll
\left(P(\log Q)^{-A}+P^{\frac{3}{2}}Q^{-\frac{1}{2}}\right)(\log Q)^{C+5}
\]
provided
\begin{equation}
\label{v_prime:UV_range}
P\ge\exp(B(\log Q)^{\frac{2}{3}}(\log\log Q)^{\frac{1}{3}}),
\end{equation}
where the implicit constant depends only on $A$ and $C$.
\end{lemma}
%%%%%%%%%%%%%%%%%%%%%%%%%%%%%%%%%%%%%%%%
\begin{proof}
Let $B_1(A)$ be the constant $B(A)$ in \cref{Lem:prime}
and for our current proof, take $B=B_1(A)$.
By \cref{V1}, we may assume $P\le Q$ since otherwise
$P^{\frac{3}{2}}Q^{-\frac{1}{2}}\gg P$.
We may also assume that there exists a prime $p$
such that $P<p\le P'$. Suppose that the prime numbers $p$ with $P<p\le P'$
are given by $q_1<\cdots<q_N$.
Then we have
\begin{equation}
\label{v_prime:indexed}
\sum_{P<p\le P'}v(p)e\left(\frac{Q}{p}\right)
=
\sum_{n=1}^{N}v(q_n)e\left(\frac{Q}{q_n}\right).
\end{equation}
We introduce
\[
S(x)=\sum_{P<p\le x}e\left(\frac{Q}{p}\right),\quad
q_0=1.
\]
Then by applying partial summation to \cref{v_prime:indexed},
\begin{align}
\sum_{P<p\le P'}v(p)e\left(\frac{Q}{p}\right)
&=
\sum_{n=1}^{N}v(q_n)\left(S(q_n)-S(q_{n-1})\right)\\
&=
\sum_{n=1}^{N-1}\left(v(q_n)-v(q_{n+1})\right)S(q_n)
+v(q_{N})S(q_{N}).
\end{align}
By \cref{V1} and \cref{V3}, this can be estimated as
\begin{equation}
\label{v_prime:sum_parts}
\begin{aligned}
&\left|\sum_{P<p\le P'}v(p)e\left(\frac{Q}{p}\right)\right|\\
&\le
\left(\sup_{P<x\le P'}|S(x)|\right)
\left(
\sum_{n=1}^{N-1}\left|v(q_{n+1})-v(q_{n})\right|
+
|v(q_{N})|
\right)\\
&\le
\left(\sup_{P<x\le P'}|S(x)|\right)
\left(
\sum_{p_n\le P'}\left|v(p_{n+1})-v(q_{n})\right|
+
|v(q_{N})|
\right)\\
&\ll
\left(\sup_{P<x\le P'}|S(x)|\right)(\log Q)^C.
\end{aligned}
\end{equation}
By \cref{v_prime:UV_range},
we can now apply \cref{Lem:prime} to the sum $S(x)$ to obtain
\[
\sup_{P<x\le P'}|S(x)|
\ll
\left(P(\log Q)^{-A}+P^{\frac{3}{2}}Q^{-\frac{1}{2}}\right)(\log Q)^{5}
\]
By substituting this estimate
into \cref{v_prime:sum_parts},
we obtain the assertion.
\end{proof}

%%%%%%%%%%%%%%%%%%%%%%%%%%%%%%%%%%%%%%%%
We next prepare a preliminary estimate.
The proof of the next lemma includes a Vinogradov-type
combinatorial decomposition finer than used by P\'etermann~\cite{Petermann}.

%%%%%%%%%%%%%%%%%%%%%%%%%%%%%%%%%%%%%%%%
\begin{lemma}
\label{Lem:v_sq_free}
Let $v(n)$ be a complex-valued multiplicative function
satisfying {\upshape\cref{V1}} and {\upshape\cref{V3}}.
For any real number $A\ge1$, there exists a real number $B=B(A)\ge1$
such that for any real numbers $P,P',Q,z\ge4$ with $P<P'\le 2P$
and any positive integer $\nu$,
\begin{equation}
\label{EQ:sum_over_prime_product}
\sum_{\substack{%
P<q\le P'\\
p_{\min}(q)>z\\
q\colon\text{\upshape square-free}\\
\omega(q)=\nu}}
v(q)e\left(\frac{Q}{q}\right)
\ll
C^{\nu}\left(P(\log Q)^{-A}+P^{\frac{3}{2}}Q^{-\frac{1}{2}}+Pz^{-1}\right)(\log Q)^{C+6}
\end{equation}
provided
\begin{equation}
\label{v_sq_free:UV_range}
P\ge \exp\left(B(\log Q)^{\frac{2}{3}}(\log\log Q)^{\frac{1}{3}}\right)
\end{equation}
where the implicit constant depends only on $A$ and $C$.
\end{lemma}
%%%%%%%%%%%%%%%%%%%%%%%%%%%%%%%%%%%%%%%%
\begin{proof}
Let $B_1(A)$ and $B_2(A)$ be the constant $B(A)$ in \cref{Lem:TypeII}
and \cref{Lem:v_prime}, respectively.
For our current lemma, we take $B=3\max(B_1(A)+1,B_2(A))$.

We may assume $Q$ is larger than some absolute constant since otherwise 
\cref{V1} implies $|v(q)|\le C^{\nu}$ if $\omega(q)=\nu$ and $q$ is square-free,
so $P(\log Q)^{-A}\gg P$ implies the assertion immediately.
Similarly, we may assume $P\le Q$
since otherwise $P^{\frac{3}{2}}Q^{-\frac{1}{2}}\gg P$.
If $\nu=1$, then we can apply \cref{Lem:v_prime}.
Thus we may further assume $\nu\ge2$.

By considering the prime factorization,
we can rewrite the sum as
\begin{equation}
\label{v_sq_free:symmetrize}
\begin{aligned}
\sum_{\substack{%
P<q\le P'\\
p_{\min}(q)>z\\
q\colon\text{square-free}\\
\omega(q)=\nu}}
v(q)e\left(\frac{Q}{q}\right)
&=
\sum_{\substack{P<p_1\cdots p_{\nu}\le P'\\p_1>\cdots>p_{\nu}>z}}
v(p_1\cdots p_{\nu})e\left(\frac{Q}{p_1\cdots p_{\nu}}\right)\\
&=
\frac{1}{\nu!}
\sum_{\substack{%
P<p_1\cdots p_{\nu}\le P'\\
p_1,\ldots,p_{\nu}>z\\
p_1,\ldots,p_{\nu}\colon\text{distinct}%
}}
v(p_1\cdots p_{\nu})e\left(\frac{Q}{p_1\cdots p_{\nu}}\right)
\end{aligned}
\end{equation}
Let us consider the sets
\begin{align}
\mathcal{P}
&=
\{(p_1,\ldots,p_\nu)\mid P<p_1\cdots p_\nu\le P',\ p_1,\dots,p_\nu>z\},\\
\mathcal{Q}
&=
\{(p_1,\ldots,p_\nu)\in\mathcal{P}\mid p_1,\cdots,p_{\nu}\colon\text{distinct}\}
\end{align}
and
\[
\mathcal{R}_{ij}
=
\{(p_1,\ldots,p_\nu)\in\mathcal{P}\mid p_i=p_j\}\quad\text{for}\ 
1\le i<j\le\nu.
\]
We define a completely multiplicative function $w(n)$ by $w(p)=v(p)$.
By \cref{V1}, we have $|w(q)|\le C^\nu$
for positive integers $q$ with $\Omega(q)=\nu$.
Furthermore, for a given finite set $\mathcal{T}$ of $\nu$-tuples of primes,
we define $S(\mathcal{T})$ by
\[
S(\mathcal{T})
=
\sum_{(p_1,\ldots,p_{\nu})\in\mathcal{T}}
w(p_1\cdots p_{\nu})e\left(\frac{Q}{p_1\cdots p_{\nu}}\right).
\]
Then the equation \cref{v_sq_free:symmetrize} implies that
\begin{equation}
\label{v_sq_free:separation}
\sum_{\substack{%
P<q\le P'\\
p_{\min}(q)>z\\
q\colon\text{square-free}\\
\omega(q)=\nu}}
v(q)e\left(\frac{Q}{q}\right)
=
\frac{1}{\nu!}
S(\mathcal{Q})
=
\frac{1}{\nu!}
S(\mathcal{P})
+
O\left(\frac{C^{\nu}}{\nu!}|\mathcal{P}\setminus\mathcal{Q}|\right).
\end{equation}

By definition, we have a covering
\begin{equation}
\label{v_sq_free:covering}
\mathcal{P}\setminus\mathcal{Q}
\subset
\bigcup_{1\le i<j\le \nu}\mathcal{R}_{ij}.
\end{equation}
By symmetry, we have $|\mathcal{R}_{ij}|=|\mathcal{R}_{12}|$
for every $i$ and $j$ with $1\le i<j\le \nu$.
By the covering \cref{v_sq_free:covering}, we find that
\begin{equation}
\label{v_sq_free:remove_distinct}
\begin{aligned}
\frac{C^{\nu}}{\nu!}
|\mathcal{P}\setminus\mathcal{Q}|
&\le
\frac{C^{\nu}}{\nu!}
\sum_{1\le i<j\le \nu}|\mathcal{R}_{ij}|
=
\frac{C^{\nu}}{\nu!}
\binom{\nu}{2}|\mathcal{R}_{12}|\\
&\le
\frac{C^{\nu}}{2(\nu-2)!}\sum_{z<p\le (2P)^{\frac{1}{2}}}
\sum_{P/p^2<p_3\cdots p_{\nu}\le P'/p^2}1.
\end{aligned}
\end{equation}
For a given integer $n$ with $P/p^2<n\le P'/p^2$,
there are at most $(\nu-2)!$ ways to express $n$ as $n=p_3\cdots p_{\nu}$.
Therefore, the estimate \cref{v_sq_free:remove_distinct}
can be continued as
\begin{equation}
\label{v_sq_free:remainder}
\begin{aligned}
\frac{C^{\nu}}{\nu!}
|\mathcal{P}\setminus\mathcal{Q}|
&\le
\frac{C^{\nu}}{2}
\sum_{z<p\le (2P)^{\frac{1}{2}}}
\sum_{P/p^2<n\le P'/p^2}1\\
&\ll
C^{\nu}P\sum_{z<p\le(2P)^{\frac{1}{2}}}\frac{1}{p^2}
\ll
C^{\nu}Pz^{-1}.
\end{aligned}
\end{equation}

Let us next consider the sets
\begin{align}
\mathcal{P}_{r}(\mathrm{I})
&=
\{(p_1,\ldots,p_\nu)\in\mathcal{P}
\mid p_{r}>P^{\frac{1}{3}},\ p_1,\ldots,p_{r-1}\le P^{\frac{1}{3}}\}\quad\text{for}\ 
1\le r\le\nu,\\
\mathcal{P}(\mathrm{II})
&=
\{(p_1,\ldots,p_\nu)\in\mathcal{P}\mid p_1,\ldots,p_{\nu}\le P^{\frac{1}{3}}\}.
\end{align}
Then we have a decomposition
\[
\mathcal{P}
=\bigcup_{r=1}^{\nu}\mathcal{P}_r(\mathrm{I})\cup\mathcal{P}(\mathrm{II}),\quad
\mathcal{P}_1(\mathrm{I}),\ldots,\mathcal{P}_{\nu}(\mathrm{I}),
\mathcal{P}(\mathrm{II})\colon\text{disjoint}.
\]
Therefore, the first term on the most right-hand side
of \cref{v_sq_free:separation} can be decomposed as
\begin{equation}
\label{v_sq_free:decomp_I_II}
\frac{1}{\nu!}
S(\mathcal{P})
=
\frac{1}{\nu!}\sum_{r=1}^{\nu}S(\mathcal{P}_{r}(\mathrm{I}))
+
\frac{1}{\nu!}S(\mathcal{P}(\mathrm{II})).
\end{equation}
By symmetry among $p_1,\ldots,p_\nu$, we find that
\begin{equation}
\label{v_sq_free:TypeI}
\frac{1}{\nu!}
S(\mathcal{P}_{r}(\mathrm{I}))
=
\frac{1}{\nu!}
\sum_{d\le P'/P^{\frac{1}{3}}}w_{r}(d)
\sum_{\substack{\max(P/d,P^{\frac{1}{3}},z)<p\le P'/d}}
v(p)e\left(\frac{Q/d}{p}\right),
\end{equation}
where the arithmetic function $w_{r}(d)$ is defined by
\begin{equation}
\label{v_sq_free:w}
w_{r}(d)
=
\sum_{\substack{%
p_1\cdots p_{\nu-1}=d\\
p_1,\ldots,p_{\nu-1}>z\\
p_1,\ldots,p_{r-1}\le P^{\frac{1}{3}}}}
w(p_1\cdots p_{\nu-1}).
\end{equation}
Since for a given integer $d$, there are at most $(\nu-1)!$ ways to express $d$
in the form $d=p_1\cdots p_{\nu-1}$, by \cref{V1}, we see that
\begin{equation}
\label{v_sq_free:w_estimate}
\left|w_{r}(d)\right|\le (\nu-1)!C^{\nu-1}.
\end{equation}
We apply \cref{Lem:v_prime} to the inner sum of \cref{v_sq_free:TypeI}.
By \cref{v_sq_free:UV_range} and our choice of $B$,
\[
\max(P/d,P^{\frac{1}{3}},z)
\ge P^{\frac{1}{3}}
\ge\exp(B_2(A)(\log Q)^{\frac{2}{3}}(\log\log Q)^{\frac{1}{3}}).
\]
Also, we find that
\[
P'/d\le2P/d\le2\max(P/d,P^{\frac{1}{3}},z),\quad
Q/d\ge \tfrac{1}{2}QP^{-\frac{2}{3}}\ge \tfrac{1}{2}Q^{\frac{1}{3}}\ge 4
\]
for large $Q$ since $P\le Q$. Therefore, we may apply \cref{Lem:v_prime}
and use \cref{v_sq_free:w_estimate} to obtain
\begin{align}
\frac{1}{\nu!}
S(\mathcal{P}_{r}(\mathrm{I}))
&\ll
\frac{C^{\nu}}{\nu}\sum_{d\le P'/P^{\frac{1}{3}}}
\left(\frac{P}{d}(\log Q)^{-A}+\frac{P^{\frac{3}{2}}Q^{-\frac{1}{2}}}{d}\right)
(\log Q)^{C+5}\\
&\ll
\frac{C^{\nu}}{\nu}
\left(P(\log Q)^{-A}+P^{\frac{3}{2}}Q^{-\frac{1}{2}}\right)(\log Q)^{C+6}.
\end{align}
This implies
\begin{equation}
\label{v_sq_free:TypeI_estimate}
\frac{1}{\nu!}\sum_{r=1}^{\nu}S(\mathcal{P}_{r}(\mathrm{I}))
\ll
C^{\nu}\left(P(\log Q)^{-A}+P^{\frac{3}{2}}Q^{-\frac{1}{2}}\right)(\log Q)^{C+6}.
\end{equation}

We move on to the sum over $\mathcal{P}(\mathrm{II})$.
We further introduce the sets
\[
\mathcal{P}_{r}(\mathrm{II})
=
\left\{(p_1,\ldots,p_{\nu})\in\mathcal{P}(\mathrm{II})\midmid
p_1\cdots p_{r}> P^{\frac{1}{3}},\ 
p_1\cdots p_{r-1}\le P^{\frac{1}{3}}
\right\}\quad\text{for}\ 2\le r\le\nu.
\]
We then obtain a decomposition
\begin{equation}
\label{v_sq_free:TypeII_decomp}
\mathcal{P}(\mathrm{II})=\bigcup_{r=2}^{\nu-1}\mathcal{P}_{r}(\mathrm{II}),\quad
\mathcal{P}_{2}(\mathrm{II}),\ldots,\mathcal{P}_{\nu-1}(\mathrm{II})\colon\text{disjoint}
\end{equation}
since if $(p_1,\ldots,p_\nu)\in\mathcal{P}(\mathrm{II})$,
then $p_1\cdots p_\nu>P>P^{\frac{1}{3}}$
and $p_1,p_{\nu}\le P^{\frac{1}{3}}$ by definition.
This gives
\begin{equation}
\label{v_sq_free:decomp_II}
\frac{1}{\nu!}S(\mathcal{P}(\mathrm{II}))
=
\frac{1}{\nu!}\sum_{r=2}^{\nu-1}S(\mathcal{P}_{r}(\mathrm{II})).
\end{equation}
For each $r$, we change the variable by
\[
u=p_1\cdots p_r,\quad
v=p_{r+1}\cdots p_{\nu}.
\]
Note that the definition of $\mathcal{P}_{r}(\mathrm{II})$ implies
\[
u=p_1\cdots p_r=(p_1\cdots p_{r-1})\cdot p_r
\le P^{\frac{1}{3}}\cdot P^{\frac{1}{3}}= P^{\frac{2}{3}}.
\]
Then the sum $S(\mathcal{P}_{r}(\mathrm{II}))$ can be expressed as
\begin{equation}
\label{v_sq_free:TypeII}
S(\mathcal{P}_{r}(\mathrm{II}))
=
\sum_{\substack{P<uv\le P'\\P^{\frac{1}{3}}<u\le P^{\frac{2}{3}}}}
\alpha_{r}(u)\beta_{r}(v)e\left(\frac{Q}{uv}\right),
\end{equation}
where $\alpha_{r}(v)$ and $\beta_{r}(v)$ are defined by
\[
\alpha_{r}(u)
=
\sum_{\substack{%
p_1\cdots p_r=u\\
z<p_1,\ldots,p_r\le P^{\frac{1}{3}}\\
p_1\cdots p_{r-1}\le P^{\frac{1}{3}}
}}w(p_1\cdots p_r),\quad
\beta_{r}(v)
=
\sum_{\substack{%
p_{r+1}\cdots p_{\nu}=v\\
z<p_{r+1},\ldots,p_{\nu}\le P^{\frac{1}{3}}
}}w(p_{r+1}\cdots p_{\nu}).
\]
Similarly to the estimate \cref{v_sq_free:w_estimate}, we see that
\begin{equation}
\label{v_sq_free:alpha_beta}
|\alpha_{r}(u)|\le r!C^r,\quad
|\beta_{r}(v)|\le (\nu-r)!C^{\nu-r}.
\end{equation}
We now employ dyadic subdivision in \cref{v_sq_free:TypeII}.
This gives
\begin{equation}
\label{v_sq_free:TypeII_dissect}
S(\mathcal{P}_{r}(\mathrm{II}))
\ll
(\log Q)^2\sup|T_{r}(U,U',V,V')|,
\end{equation}
where
\[
T_{r}(U,U',V,V')
=
\sum_{\substack{P<uv\le P'\\U<u\le U'\\V<v\le V'}}
\alpha_{r}(u)\beta_{r}(v)e\left(\frac{Q}{uv}\right)
\]
and the supremum is taken over real numbers $U,U',V,V'$ in the range
\[
P/4<UV\le2P,\quad
U<U'\le 2U,\quad
V<V'\le 2V,\quad
P^{\frac{1}{3}}\le U\le P^{\frac{2}{3}}.
\]
In this range, we have $V>P/4U\ge P^{\frac{1}{3}}/4$ so that
\[
U,V
\ge P^{\frac{1}{3}}/4
\ge\exp(B_1(A)(\log Q)^{\frac{2}{3}}(\log\log Q)^{\frac{1}{3}})
\]
for large $Q$ by \cref{v_sq_free:UV_range} and our choice of $B$.
Therefore, we can apply \cref{Lem:TypeII} to the sum $T_{r}(U,U',V,V')$.
By combining with \cref{v_sq_free:alpha_beta}, \cref{Lem:TypeII} gives
\[
T_{r}(U,U',V,V')
\ll
r!(\nu-r)!C^{\nu}\left(P(\log Q)^{-A}+P^{\frac{3}{2}}Q^{-\frac{1}{2}}\right)
(\log Q).
\]
Substituting this estimate into \cref{v_sq_free:TypeII_dissect},
we obtain
\begin{align}
\frac{1}{\nu!}S(\mathcal{P}_{r}(\mathrm{II}))
&\ll
\frac{1}{\binom{\nu}{r}}
C^{\nu}\left(P(\log Q)^{-A}+P^{\frac{3}{2}}Q^{-\frac{1}{2}}\right)(\log Q)^{3}\\
&\ll
\frac{1}{\nu}
C^{\nu}\left(P(\log Q)^{-A}+P^{\frac{3}{2}}Q^{-\frac{1}{2}}\right)(\log Q)^{3}
\end{align}
provided $2\le r\le \nu-1$. By combining with \cref{v_sq_free:decomp_II}, this implies
\begin{equation}
\label{v_sq_free:TypeII_estimate}
\frac{1}{\nu!}
S(\mathcal{P}(\mathrm{II}))
\ll
C^{\nu}\left(P(\log Q)^{-A}+P^{\frac{3}{2}}Q^{-\frac{1}{2}}\right)(\log Q)^{3}.
\end{equation}
By
\cref{v_sq_free:separation},
\cref{v_sq_free:remainder},
\cref{v_sq_free:decomp_I_II},
\cref{v_sq_free:TypeI_estimate}
and
\cref{v_sq_free:TypeII_estimate},
we arrive at the assertion.
\end{proof}

%%%%%%%%%%%%%%%%%%%%%%%%%%%%%%%%%%%%%%%%
Before proceeding to the exponential sums with the multiplicative
functions, we further prepare two lemmas.
The first one provides some estimates for $v(n)$.

%%%%%%%%%%%%%%%%%%%%%%%%%%%%%%%%%%%%%%%%
\begin{lemma}
\label{Lem:v_L1}
Let $v(n)$ be a complex-valued multiplicative function satisfying
{\upshape\cref{V2}}.
Then for $x\ge 4$, we have
\[
\sum_{n\le x}|v(n)|\ll x(\log x)^C,\quad
\sum_{n\le x}\frac{|v(n)|}{n}\ll (\log x)^{C},
\]
where the implicit constants depend only on $C$.
\end{lemma}
%%%%%%%%%%%%%%%%%%%%%%%%%%%%%%%%%%%%%%%%
\begin{proof}
By the Cauchy--Schwarz inequality and \cref{V2},
\begin{equation}
\label{v_L1:first}
\sum_{n\le x}|v(n)|
\ll
x^{\frac{1}{2}}\left(\sum_{n\le x}|v(n)|^2\right)^{\frac{1}{2}}
\ll
x(\log x)^{\frac{C}{2}}.
\end{equation}
This proves the former estimate.
We dissect the latter sum dyadically to obtain
\begin{equation}
\label{v_L1:dissect}
\sum_{n\le x}\frac{|v(n)|}{n}
\ll
(\log x)\sup_{1\le U\le x}U^{-1}\sum_{U\le n\le 2U}|v(n)|.
\end{equation}
By substituting \cref{v_L1:first} here,
we obtain the latter estimate since $C\ge2$.
\end{proof}

%%%%%%%%%%%%%%%%%%%%%%%%%%%%%%%%%%%%%%%%
The next one is a well-known estimate
for the number of smooth numbers.

%%%%%%%%%%%%%%%%%%%%%%%%%%%%%%%%%%%%%%%%
\begin{lemma}
\label{Lem:de_Bruijn}
For a sufficiently large real number $x$ and a real number $y$,
we have
\[
\psi(x,y)
\le
x\exp\left(-\frac{1}{2}\frac{\log x}{\log y}\log\frac{\log x}{\log y}\right)
\]
provided
\begin{equation}
\label{de_Bruijn:y_range}
(\log x)^2\le y\le\exp\left(\frac{\log x}{\log\log x}\right).
\end{equation}
\end{lemma}
%%%%%%%%%%%%%%%%%%%%%%%%%%%%%%%%%%%%%%%%
\begin{proof}
Let
\[
u=\frac{\log x}{\log y}.
\]
Then \cref{de_Bruijn:y_range} implies
\begin{equation}
\label{de_Bruijn:u_size}
u\ge\log\log x
\end{equation}
so that $u\to\infty$ as $x\to\infty$.
By Theorem~7.6 of \cite{Montgomery_Vaughan},
\begin{equation}
\label{de_Bruijn:pre}
\psi(x,y)
<x\exp\left(-(1+o(1))u\log u
+\log\log y
+O\left(\frac{u^2\log u}{y}\right)\right)
\end{equation}
as $x\to\infty$. By \cref{de_Bruijn:u_size}, it follows that
\[
\frac{\log\log y}{u\log u}
\le
\frac{1}{\log u}
\to0
\quad\text{as $x\to\infty$}
\]
and by \cref{de_Bruijn:y_range},
\[
\frac{u}{y}\le\frac{1}{\log x}\to0
\quad\text{as $x\to\infty$}.
\]
Therefore,
\[
\log\log y=o(u\log u),\quad
\frac{u^2\log u}{y}=o(u\log u)\quad
(x\to\infty).
\]
On inserting these estimate into \cref{de_Bruijn:pre},
we arrive at the lemma.
\end{proof}

%%%%%%%%%%%%%%%%%%%%%%%%%%%%%%%%%%%%%%%%
After the above preparations, we can now prove the exponential sum estimate
with our multiplicative function $v(n)$.

%%%%%%%%%%%%%%%%%%%%%%%%%%%%%%%%%%%%%%%%
\begin{lemma}
\label{Lem:v}
Let $v(n)$ be a complex-valued multiplicative function satisfying
{\upshape\cref{V1}}, {\upshape\cref{V2}}, and {\upshape\cref{V3}}.
For any real number $A\ge1$, there exists a real number $B=B(A)\ge1$
such that for any real numbers $P,P',Q\ge 4$ with $P<P'\le2P$,
we have
\[
\sum_{P<n\le P'}v(n)e\left(\frac{Q}{n}\right)
\ll
\left(P(\log Q)^{-A}+P^{\frac{3}{2}}Q^{-\frac{1}{2}}\right)(\log Q)^{6C+6}
\]
provided
\begin{equation}
\label{v:UV_range}
P\ge \exp(B(\log Q)^{\frac{2}{3}}(\log\log Q)^{\frac{1}{3}}),
\end{equation}
where the implicit constant depends only on $A$ and $C$.
\end{lemma}
%%%%%%%%%%%%%%%%%%%%%%%%%%%%%%%%%%%%%%%%
\begin{proof}
Let $B_1(A)$ be the constant $B(A)$ in \cref{Lem:v_sq_free}
and we take $B=2B_1(A)$ for the current proof.
We may assume that $Q$ is larger than some constant
depending only on $A$ since otherwise \cref{Lem:v_L1} and $P(\log Q)^{-A}\gg P$
imply the assertion trivially.
Also, we may assume $2P\le Q$ since otherwise 
\cref{Lem:v_L1} and $P^{\frac{3}{2}}Q^{-\frac{1}{2}}\gg P$
gives the assertion.

Let
\begin{equation}
\label{v:z}
z:=\exp\left(\frac{\log P}{2\log\log Q}\right)
\end{equation}
so that
\begin{equation}
\label{v:z_large}
z
\ge\exp\left(\frac{B}{2}\left(\frac{\log Q}{\log\log Q}\right)^{\frac{2}{3}}\right)
\gg(\log Q)^{2A}
\end{equation}
by \cref{v:UV_range}.
For an integer $n$ with $P<n\le P'$, we have a unique expression
\begin{equation}
\label{v:n_mq}
n=mq,\quad
P<mq\le P',\quad
p_{\max}(m)\le z,\quad
p_{\min}(q)>z.
\end{equation}
Therefore, by using this expression for the change of variables, we have
\begin{equation}
\label{v:smooth_rough_decomp}
\begin{aligned}
\sum_{P<n\le P'}v(n)e\left(\frac{Q}{n}\right)
&=
\sum_{\substack{P<mq\le P'\\p_{\max}(m)\le z\\p_{\min}(q)>z}}
v(m)v(q)e\left(\frac{Q}{mq}\right)\\
&=
\sum_{m\le P^{\frac{1}{2}}}
+
\sum_{m> P^{\frac{1}{2}}}
=
{\sum}_{1}+{\sum}_{2},\quad\text{say}
\end{aligned}
\end{equation}
since the conditions $p_{\max}(m)\le z$ and $p_{\min}(q)> z$ imply $(m,q)=1$.

The sum ${\sum}_1$ can be expressed as
\begin{equation}
\label{v:sum1_double_sum}
{\sum}_1
=
\sum_{\substack{m\le P^{\frac{1}{2}}\\p_{\max}(m)\le z}}v(m)
\sum_{\substack{P/m<q\le P'/m\\p_{\min}(q)>z}}v(q)e\left(\frac{Q/m}{q}\right).
\end{equation}
We now classify $q$ in the inner sum
according to whether or not $q$ is square-free, i.e.~%
we introduce a decomposition
\begin{equation}
\label{v:sum1_inner_decomp}
\begin{aligned}
\sum_{\substack{P/m<q\le P'/m\\p_{\min}(q)>z}}v(q)e\left(\frac{Q/m}{q}\right)
&=
\sum_{q\colon\text{square-free}}
+
\sum_{q\colon\text{not square-free}}\\
&=
{\sum}_{11}+{\sum}_{12},\quad\text{say}.
\end{aligned}
\end{equation}
For the former sum ${\sum}_{11}$, we apply \cref{Lem:v_sq_free}.
We first classify $q$ according to the value of $\omega(q)$.
Since $p_{\min}(q)>z$, we have
\[
P'\ge q>z^{\omega(q)}
\]
so that
\[
\omega(q)
<\frac{\log P'}{\log z}
\le\frac{2\log P}{\log z}
=4\log\log Q.
\]
Also, since $q>P/m\ge P^{\frac{1}{2}}>1$ provided $m\le P^{\frac{1}{2}}$,
we have $\omega(q)\ge1$.
Therefore,
\begin{equation}
\label{v:sum11_omega_decomp}
{\sum}_{11}
=
\sum_{\nu=1}^{[4\log\log Q]}
\sum_{\substack{%
P/m<q\le P'/m\\
p_{\min}(q)>z\\
q\colon\text{square-free}\\
\omega(q)=\nu}}
v(q)e\left(\frac{Q/m}{q}\right)
\end{equation}
We now apply \cref{Lem:v_sq_free} to the inner sum.
Since
\[
P/m
\ge P^{\frac{1}{2}}
\ge\exp(B_1(A)(\log Q)^{\frac{2}{3}}(\log\log Q)^{\frac{1}{3}}),\quad
Q/m\ge QP^{-\frac{1}{2}}\ge Q^{\frac{1}{2}}\ge4
\]
for large $Q$ provided $m\le P^{\frac{1}{2}}$,
we may apply \cref{Lem:v_sq_free} to obtain
\begin{equation}
\begin{aligned}
\sum_{\substack{%
P/m<q\le P'/m\\
p_{\min}(q)>z\\
q\colon\text{square-free}\\
\omega(q)=\nu}}
v(q)e\left(\frac{Q/m}{q}\right)
&\ll
C^{\nu}\left(\frac{P}{m}(\log Q)^{-A}
+\frac{P^{\frac{3}{2}}Q^{-\frac{1}{2}}}{m}
+\frac{P}{m}z^{-1}\right)
(\log Q)^{C+6}\\[-5mm]
&\ll
e^{C\nu}
\left(\frac{P}{m}(\log Q)^{-A}+\frac{P^{\frac{3}{2}}Q^{-\frac{1}{2}}}{m}\right)
(\log Q)^{C+6},
\end{aligned}
\end{equation}
where we used \cref{v:z_large}. By substituting this estimate into
\cref{v:sum11_omega_decomp} and by using
\[
\sum_{\nu=1}^{[4\log\log Q]}e^{C\nu}
\le
e^{4C\log\log Q}\left(1+e^{-C}+e^{-2C}+\cdots\right)
\ll
(\log Q)^{4C},
\]
we obtain
\begin{equation}
\label{v:sum11_estimate}
{\sum}_{11}
\ll
\left(\frac{P}{m}(\log Q)^{-A}+\frac{P^{\frac{3}{2}}Q^{-\frac{1}{2}}}{m}\right)
(\log Q)^{5C+6}
\end{equation}
For the latter sum ${\sum}_{12}$, by the Cauchy-Schwarz inequality and \cref{V2},
\begin{align}
\left({\sum}_{12}\right)^2
&\ll
\frac{P(\log P)^C}{m}
\sum_{\substack{%
P/m<q\le P'/m\\
p_{\min}(q)>z\\
q\colon\text{not square-free}}}1\\
&\ll
\frac{P(\log P)^C}{m}
\sum_{z<p\le(2P)^{\frac{1}{2}}}
\sum_{P/mp^2<q\le P'/mp^2}1\\
&\ll
\left(\frac{P(\log P)^C}{m}\right)^2
\sum_{z<p\le(2P)^{\frac{1}{2}}}\frac{1}{p^2}
\ll
\frac{1}{z}\left(\frac{P(\log P)^C}{m}\right)^2
\end{align}
Thus, by using \cref{v:z_large},
\begin{equation}
\label{v:sum12_estimate}
{\sum}_{12}
\ll
\frac{P(\log P)^C}{m}z^{-\frac{1}{2}}
\ll
\frac{P}{m}(\log Q)^{C-A}.
\end{equation}
By \cref{v:sum1_inner_decomp}, \cref{v:sum11_estimate} and \cref{v:sum12_estimate},
we obtain
\[
\sum_{\substack{P/m<q\le P'/m\\p_{\min}(q)>z}}v(q)e\left(\frac{Q/m}{q}\right)
\ll
\left(\frac{P}{m}(\log Q)^{-A}+\frac{P^{\frac{3}{2}}Q^{-\frac{1}{2}}}{m}\right)
(\log Q)^{5C+6}
\]
By substituting this into \cref{v:sum1_double_sum}
and by using \cref{Lem:v_L1},
we obtain
\begin{equation}
\label{v:sum1_estimate}
{\sum}_1
\ll
\left(P(\log Q)^{-A}+P^{\frac{3}{2}}Q^{-\frac{1}{2}}\right)(\log Q)^{6C+6}.
\end{equation}

We next estimate ${\sum}_2$.
We first estimate trivially
\[
{\sum}_2
\ll
\sum_{\substack{%
P<mq\le P'\\
p_{\max}(m)\le z\\
p_{\min}(q)>z\\
m>P^{\frac{1}{2}}}}|v(mq)|.
\]
Then by the Cauchy--Schwarz inequality,
\begin{equation}
\label{v:sum2_CS}
\left({\sum}_2\right)^2
\ll
\Bigg(
\sum_{\substack{%
P<mq\le P'\\
p_{\max}(m)\le z\\
p_{\min}(q)>z}}|v(mq)|^2
\Bigg)
\Bigg(
\sum_{\substack{%
P<mq\le P'\\
p_{\max}(m)\le z\\
p_{\min}(q)>z\\
m>P^{\frac{1}{2}}}}1
\Bigg)
=
\left({\sum}_{21}\right)\left({\sum}_{22}\right),\quad\text{say}.
\end{equation}
For the former sum ${\sum}_{21}$,
we recall our change of variables \cref{v:n_mq}
and trace back the substitution to obtain
\begin{equation}
\label{v:sum21_estimate}
{\sum}_{21}
=
\sum_{P<n\le P'}|v(n)|^2
\ll
P(\log P)^{C},
\end{equation}
where we used \cref{V2}. 
For the latter sum ${\sum}_{22}$, we use \cref{Lem:de_Bruijn}
with the parameters $P^{\frac{1}{2}}<x\le P'$ and $y=z$.
We can check \cref{de_Bruijn:y_range}
with this choice of parameters as
\[
(\log x)^2\le (\log P')^2\le (\log Q)^2
\le z=\exp\left(\frac{\log P}{2\log\log Q}\right)
\le\exp\left(\frac{\log x}{\log\log x}\right).
\]
Thus, we may apply \cref{Lem:de_Bruijn} to obtain
\begin{align}
\psi(x,z)
&\ll
x\exp\left(-\frac{1}{2}\frac{\log x}{\log z}\log\frac{\log x}{\log z}\right)\\
&\ll
x\exp\left(-\frac{1}{4}\frac{\log P}{\log z}\log\frac{\log P}{2\log z}\right)\\
&\ll
x\exp\left(-\frac{1}{2}\log\log Q\log\log\log Q\right)
\ll
x(\log Q)^{-2A}
\ll
P(\log Q)^{-2A}
\end{align}
for $P^{\frac{1}{2}}\le x\le P'$
provided $Q$ is larger than some constant depending on $A$.
Thus,
\begin{equation}
\label{v:sum22_estimate}
{\sum}_{22}
=
\sum_{\substack{q\le P'/P^{\frac{1}{2}}\\p_{\min}(q)>z}}
\sum_{\substack{P/q<m\le P'/q\\p_{\max}(m)\le z}}1
\ll
P(\log Q)^{-2A}\sum_{q\le P'/P^{\frac{1}{2}}}\frac{1}{q}
\ll
P(\log Q)^{-2A+1}.
\end{equation}
Combining this estimate with \cref{v:sum2_CS} and \cref{v:sum21_estimate},
we obtain
\begin{equation}
\label{v:sum2_estimate}
{\sum}_{2}
\ll
P(\log Q)^{C-A}.
\end{equation}
By substituting \cref{v:sum1_estimate} and \cref{v:sum2_estimate}
into \cref{v:smooth_rough_decomp}, we arrive at the lemma.
\end{proof}

%%%%%%%%%%%%%%%%%%%%%%%%%%%%%%%%%%%%%%%%
\section{Error term estimate}
In this section, we prove \cref{Thm:v_thm}
by using the exponential sum estimate obtained in \cref{Section:v}.
We start with a standard translation of sums with $\psi(x)$
to exponential sums.

%%%%%%%%%%%%%%%%%%%%%%%%%%%%%%%%%%%%%%%%
\begin{lemma}
\label{Lem:Vaaler}
Let $(x_n)_{n\in I}$ be a sequence of real numbers
defined over integers in a set of integers $I$,
$g(n)$ be a complex-valued function defined on $I$,
$G(n)$ be a positive function defined on $I$ such that
\begin{equation}
\label{Vaaler:G}
|g(n)|\le G(n)
\end{equation}
for all integer $n\in I$, and $H\ge1$ be a real number. Then
\[
\left|\sum_{n\in I}g(n)\psi(x_n)\right|
\ll
\sum_{0<|h|\le H}\frac{1}{|h|}\left|\sum_{n\in I}g(n)e(hx_n)\right|
+\frac{1}{H}\sum_{0\le h\le H}\left|\sum_{n\in I}G(n)e(hx_n)\right|,
\]
where the implicit constant is absolute.
\end{lemma}
%%%%%%%%%%%%%%%%%%%%%%%%%%%%%%%%%%%%%%%%
\begin{proof}
Let $N=[H]\ge1$. Then we use Vaaler's approximation
\begin{equation}
\label{Vaaler:approximant}
\psi_{N}(x)=\sum_{0<|h|\le N}c_{N}(h)e(hx),
\end{equation}
\begin{equation}
\label{Vaaler:error}
|\psi(x)-\psi_{N}(x)|
\le\frac{1}{2(N+1)}\sum_{|h|\le N}\left(1-\frac{|h|}{N+1}\right)e(hx),
\end{equation}
where the complex coefficients $c_{N}(h)$ satisfies
\begin{equation}
\label{Vaaler:coefficient}
|c_{N}(h)|\le\frac{1}{2\pi|h|}.
\end{equation}
For the details and the proof of this approximation,
see \cite[p.210, Theorem~18]{Vaaler} or \cite[p.116, Theorem A.6]{Graham_Kolesnik}.
Then, we decompose the sum as
\begin{equation}
\label{Vaaler:decomposition}
\sum_{n\in I}g(n)\psi(x_n)
=
\sum_{n\in I}g(n)\psi_{N}(x_n)+\sum_{n\in I}g(n)\left(\psi(x_n)-\psi_{N}(x_n)\right).
\end{equation}
For the former sum, by \cref{Vaaler:approximant} and \cref{Vaaler:coefficient},
we have
\begin{equation}
\label{Vaaler:g_part}
\begin{aligned}
\Bigg|\sum_{n\in I}g(n)\psi_{N}(x_n)\Bigg|
&=
\Bigg|\sum_{0<|h|\le N}c_{N}(h)\sum_{n\in I}g(n)e(hx_n)\Bigg|\\
&\ll
\sum_{0<|h|\le H}\frac{1}{|h|}\Bigg|\sum_{n\in I}g(n)e(hx_n)\Bigg|.
\end{aligned}
\end{equation}
For the latter sum, by \cref{Vaaler:G},
\[
\Bigg|\sum_{n\in I}g(n)\left(\psi(x_n)-\psi_{N}(x_n)\right)\Bigg|
\le
\sum_{n\in I}G(n)|\psi(x_n)-\psi_{N}(x_n)|.
\]
Then we use \cref{Vaaler:error} to obtain
\begin{equation}
\label{Vaaler:pre_G_part}
\left|\sum_{n\in I}g(n)\left(\psi(x_n)-\psi_{N}(x_n)\right)\right|
\le
\frac{1}{2(N+1)}\sum_{|h|\le N}\Bigg|\sum_{n\in I}G(n)e(hx_n)\Bigg|.
\end{equation}
Since $G(n)$ is real valued, by taking the complex conjugate, we find that
\[
\Bigg|\sum_{n\in I}G(n)e(hx_n)\Bigg|=\Bigg|\sum_{n\in I}G(n)e(-hx_n)\Bigg|.
\]
Thus, by \cref{Vaaler:pre_G_part}, we have
\begin{equation}
\label{Vaaler:G_part}
\left|\sum_{n\in I}g(n)\left(\psi(x_n)-\psi_{N}(x_n)\right)\right|
\ll
\frac{1}{H}\sum_{0\le h\le H}\Bigg|\sum_{n\in I}G(n)e(hx_n)\Bigg|.
\end{equation}
Combining \cref{Vaaler:decomposition}, \cref{Vaaler:g_part} and \cref{Vaaler:G_part},
we arrive at the lemma.
\end{proof}

%%%%%%%%%%%%%%%%%%%%%%%%%%%%%%%%%%%%%%%%
By \cref{Lem:Vaaler},
we can now translate the exponential estimate given in \cref{Lem:v}
to the estimate of the sum involving $\psi(x)$.

%%%%%%%%%%%%%%%%%%%%%%%%%%%%%%%%%%%%%%%%
\begin{lemma}
\label{Lem:v_psi}
Let $v(n)$ be a complex-valued multiplicative function satisfying
{\upshape\cref{V1}}, {\upshape\cref{V2}}, and {\upshape\cref{V3}}.
For any real number $A\ge1$, there exists a real number $B=B(A,C)\ge1$
such that for any real numbers $P,P',Q\ge 4$ with $P<P'\le2P$,
we have
\[
\sum_{P<n\le P'}v(n)\psi\left(\frac{Q}{n}\right)
\ll
P(\log Q)^{-A}
\]
provided
\begin{equation}
\label{v_psi:UV_range}
\exp(B(\log Q)^{\frac{2}{3}}(\log\log Q)^{\frac{1}{3}})
\le P
\le Q(\log Q)^{-2A-12C-14}.
\end{equation}
where the implicit constant depends only on $A$ and $C$.
\end{lemma}
%%%%%%%%%%%%%%%%%%%%%%%%%%%%%%%%%%%%%%%%
\begin{proof}
We may assume that $Q$ is larger than some absolute constant depending on $A$ and $C$.
Let $B_1(A)$ be the constant $B(A)$ in \cref{Lem:v}
and for the current proof, we take $B(A)=2B_1(A+6C+7)$.
We use \cref{Lem:Vaaler} with
\[
x_n=\frac{Q}{n},\quad
I=(P,P']\cap\mathbb{Z},\quad
g(n)=v(n),\quad
G(n)=|v(n)|,\quad
H=(\log Q)^{A+C}.
\]
For the assumption of \cref{Lem:v}, it suffices to check that
\begin{align}
P
&\ge\exp(2B_1(A+6C+7)(\log Q)^{\frac{2}{3}}(\log\log Q)^{\frac{1}{3}})\\
&\ge\exp(B_1(A+6C+7)(\log hQ)^{\frac{2}{3}}(\log\log hQ)^{\frac{1}{3}}),
\end{align}
which holds for $1\le h\le H$ and sufficiently large $Q$.
Note that the multiplicative function $\overline{v(n)}$
also satisfies \cref{V1}, \cref{V2}, and \cref{V3}.
Thus, we may use \cref{Lem:v} with $g(n)=v(n)$ and $\overline{v(n)}$.
This gives
\begin{equation}
\label{v_psi:g_sum_estimate}
\begin{aligned}
&\sum_{0<|h|\le H}\frac{1}{|h|}\left|\sum_{P<n\le P'}v(n)e\left(\frac{hQ}{n}\right)\right|\\
&=
\sum_{1\le h\le H}\frac{1}{h}
\left|\sum_{P<n\le P'}v(n)e\left(\frac{hQ}{n}\right)\right|
+
\sum_{1\le h\le H}\frac{1}{h}
\left|\sum_{P<n\le P'}\overline{v(n)}e\left(\frac{hQ}{n}\right)\right|\\
&\ll
\left(P(\log Q)^{-A-6C-7}+P^{\frac{3}{2}}Q^{-\frac{1}{2}}\right)(\log Q)^{6C+7}
\ll
P(\log Q)^{-A}
\end{aligned}
\end{equation}
by \cref{v_psi:UV_range}.
For the sum with $G(n)=|v(n)|$, we use \cref{Lem:v_L1} to the term $h=0$.
By \cref{Lem:v_L1}, this gives
\begin{equation}
\label{v_psi:G_sum}
\begin{aligned}
&\frac{1}{H}\sum_{0\le h\le H}
\left|\sum_{P<n\le P'}|v(n)|e\left(\frac{hQ}{n}\right)\right|\\
&\ll
P(\log Q)^{-A}
+
\sum_{1\le h\le H}\frac{1}{h}\left|\sum_{P<n\le P'}|v(n)|e\left(\frac{hQ}{n}\right)\right|.
\end{aligned}
\end{equation}
The function $|v(n)|$ trivially satisfies \cref{V1} and \cref{V2}
with the same value of $C$ as for $v(n)$.
For \cref{V3}, by the triangle inequality, we can see
\[
\sum_{p_n\le x}||v(p_{n+1})|-|v(p_n)||
\le
\sum_{p_n\le x}|v(p_{n+1})-v(p_n)|
\le
C(\log x)^C
\]
for $x\ge4$. Therefore we can estimate the second term on the right-hand side
of \cref{v_psi:G_sum} similarly to \cref{v_psi:g_sum_estimate}.
Therefore,
\begin{equation}
\label{v_psi:G_sum_estimate}
\frac{1}{H}\sum_{0\le h\le H}\left|\sum_{P<n\le P'}|v(n)|e\left(\frac{hQ}{n}\right)\right|
\ll
P(\log Q)^{-A}.
\end{equation}
On inserting \cref{v_psi:g_sum_estimate} and \cref{v_psi:G_sum_estimate}
into \cref{Lem:Vaaler} with our choice, we obtain the lemma.
\end{proof}

%%%%%%%%%%%%%%%%%%%%%%%%%%%%%%%%%%%%%%%%
We can now prove \cref{Thm:v_thm}.

%%%%%%%%%%%%%%%%%%%%%%%%%%%%%%%%%%%%%%%%
\begin{proof}[Proof of \cref{Thm:v_thm}]
We may assume that $x$ is larger than some constant
depending only on $\theta$, $C$, and the implicit constant in {\upshape\cref{V}}.
Let us take $B=B(1,C)$ in \cref{Lem:v_psi}
and let
\[
z=\exp(B(\log x)^{\frac{2}{3}}(\log\log x)^{\frac{1}{3}}).
\]
If $y\le z$, then the assertion immediately follows by \cref{V}.
Thus we may assume $z<y$.
Then we dissect the sum at $n=z$ as
\begin{equation}
\label{v_thm:decomp}
\sum_{n\le y}\frac{v(n)}{n}\psi\left(\frac{x}{n}\right)
=
\sum_{n\le z}+\sum_{z<n\le y}
=
{\sum}_1+{\sum}_2,\quad\text{say}.
\end{equation}
By \cref{V}, the former sum ${\sum}_1$ is just
\begin{equation}
\label{v_thm:sum1_estimate}
{\sum}_1
\ll
\sum_{n\le z}\frac{|v(n)|}{n}
\ll
(\log z)^{\kappa}
\ll
(\log x)^{\frac{2\kappa}{3}}(\log\log x)^{\frac{\kappa}{3}}.
\end{equation}
For the latter sum ${\sum}_2$, we employ dyadic subdivision
and partial summation
\begin{equation}
\label{v_thm:sum2_dissect}
\begin{aligned}
{\sum}_2
&\ll
(\log x)\sup_{\substack{z\le P\le y\\P<P'\le 2P}}
\left|\sum_{P<n\le P'}\frac{v(n)}{n}\psi\left(\frac{x}{n}\right)\right|\\
&\ll
(\log x)\sup_{\substack{z\le P\le y\\P<P'\le 2P}}
P^{-1}\left|\sum_{P<n\le P'}v(n)\psi\left(\frac{x}{n}\right)\right|.
\end{aligned}
\end{equation}
If $x$ is larger than some constant depending on $C$ and $\theta$,
then $z<P\le y$ implies
\[
\exp(B(1,C)(\log x)^{\frac{2}{3}}(\log\log x)^{\frac{1}{3}})
\le P
\le x\exp\left(-(\log x)^{\theta}\right)\le x(\log x)^{-12C-16}.
\]
Hence, we may apply \cref{Lem:v_psi} with $A=1$
to the right-hand side of \cref{v_thm:sum2_dissect}.
Then
\begin{equation}
\label{v_thm:sum2_estimate}
{\sum}_2
\ll
1.
\end{equation}
On inserting \cref{v_thm:sum1_estimate} and \cref{v_thm:sum2_estimate}
into \cref{v_thm:decomp}, we arrive at the theorem.
\end{proof}

%%%%%%%%%%%%%%%%%%%%%%%%%%%%%%%%%%%%%%%%
\section{The Balakrishnan--P\'etermann method}
\label{Section:BP_completed}
We next prove \cref{Thm:BP_completed} by using \cref{Thm:v_thm}.
Thus, in this section, we discuss under the hypothesis in \cref{Thm:BP}.
It is not clear whether or not the hypothesis of \cref{Thm:BP}
immediately implies the assumptions of \cref{Thm:v_thm}.
Indeed, the hypothesis of \cref{Thm:BP} does not assure
the multiplicativity of $v(n)$.
Therefore, we decompose $v(n)$ into $\tau_{\alpha}(n)$ and $b(n)$
by using the definition \cref{BP:v} and apply \cref{Thm:v_thm}
to the function $\tau_{\alpha}(n)$.
We start with recalling basic properties of $\tau_{\alpha}(n)$.

%%%%%%%%%%%%%%%%%%%%%%%%%%%%%%%%%%%%%%%%
\begin{lemma}
\label{Lem:tau_explicit}
For any complex number $\alpha$,
we have the Dirichlet series expansion
\begin{equation}
\label{tau_explicit:series}
\zeta(s)^{\alpha}=\sum_{n=1}^{\infty}\frac{\tau_{\alpha}(n)}{n^s},
\end{equation}
where the multiplicative function $\tau_{\alpha}(n)$ is defined by
\begin{equation}
\label{tau_explicit:explicit}
\tau_{\alpha}(1)=1,\quad
\tau_{\alpha}(p^{\nu})
=\binom{\alpha+\nu-1}{\nu}
=\frac{\prod_{\ell=1}^{\nu}(\alpha+\ell-1)}{\nu!}
\end{equation}
and the series {\upshape\cref{tau_explicit:series}}
converges absolutely for $\sigma>1$. Moreover, we have
\begin{equation}
\label{tau_explicit:bound}
\tau_{\alpha}(n)\ll n^{\epsilon}
\end{equation}
for every $\epsilon>0$, where the implicit constant depends on $\alpha$ and $\epsilon$.
\end{lemma}
%%%%%%%%%%%%%%%%%%%%%%%%%%%%%%%%%%%%%%%%
\begin{proof}
We have the Taylor expansion
\begin{equation}
\label{tau_explicit:binomial}
(1-z)^{-\alpha}
=
\sum_{\nu=0}^{\infty}\binom{\alpha+\nu-1}{\nu}z^{\nu},
\end{equation}
which has the radius of convergence $1$. Thus, for $\sigma>1$
and $p\ge2$, we have
\[
\sum_{\nu=1}^{\infty}\left|\binom{\alpha+\nu-1}{\nu}\right|\frac{1}{p^{\nu\sigma}}
\ll p^{-\sigma}.
\]
Therefore, by using the Euler product,
\[
\sum_{n=1}^{\infty}\frac{|\tau_{\alpha}(n)|}{n^s}
=
\prod_{p}\left(1+
\sum_{\nu=1}^{\infty}\left|\binom{\alpha+\nu-1}{\nu}\right|\frac{1}{p^{\nu\sigma}}\right)
=
\prod_{p}(1+O(p^{-\sigma}))<+\infty
\]
for $\sigma>1$. This proves that the series \cref{tau_explicit:series}
converges absolutely for $\sigma>1$. Then by comparing the Euler product
of the both sides of \cref{tau_explicit:series} and checking that
their argument coincides for $\sigma>1$, we obtain
\cref{tau_explicit:series} and \cref{tau_explicit:explicit}.
Finally, by \cref{tau_explicit:explicit},
\[
\left|\tau_{\alpha}(p^{\nu})\right|
\le
\prod_{\ell=1}^{\nu}\left(1+\frac{|\alpha|}{\ell}\right)
\le
\exp\left(|\alpha|\sum_{\ell=1}^{\nu}\frac{1}{\ell}\right)
\le
(3\nu)^{|\alpha|}
\le
(\nu+1)^{2|\alpha|}
\le
\tau(p^{\nu})^{2|\alpha|}.
\]
Therefore, by the well-known bound $\tau(n)\ll n^{\epsilon}$, we arrive at
\[
|\tau_{\alpha}(n)|\le\tau(n)^{2|\alpha|}\ll n^{\epsilon}.
\]
This completes the proof.
\end{proof}

%%%%%%%%%%%%%%%%%%%%%%%%%%%%%%%%%%%%%%%%
\begin{lemma}
\label{Lem:tau_condition}
For any complex number $\alpha$,
the multiplicative function $\tau_{\alpha}(n)$ satisfies
{\upshape\cref{V1}}, {\upshape\cref{V2}}, {\upshape\cref{V3}} for some constant $C$
and the estimate
\begin{equation}
\label{tau_condition:L1}
\sum_{n\le x}\frac{|\tau_{\alpha}(n)|}{n}
\ll
(\log x)^{|\alpha|}
\end{equation}
for $x\ge4$.
\end{lemma}
\begin{proof}
By \cref{Lem:tau_explicit}, we have $\tau(p)=\alpha$ for every prime $p$.
Thus the conditions {\upshape\cref{V1}}, {\upshape\cref{V3}} trivially holds.
For \cref{V2}, we start with
\begin{align}
\sum_{n\le x}|\tau_{\alpha}(n)|^2
&\le
x\sum_{n\le x}\frac{|\tau_{\alpha}(n)|^2}{n}\\
&\le
x\prod_{p\le x}
\left(1+\frac{|\tau_{\alpha}(p)|^2}{p}+\frac{|\tau_{\alpha}(p^2)|^2}{p^2}+\cdots\right)\\
&\le
x\prod_{p\le x}
\left(1+\frac{|\alpha|^2}{p}\right)
\left(1+\frac{|\tau_{\alpha}(p^2)|^2}{p^2}+\frac{|\tau_{\alpha}(p^3)|^2}{p^3}
+\cdots\right).
\end{align}
Then by using Mertens' theorem and \cref{tau_explicit:bound}, this gives
\[
\sum_{n\le x}|\tau_{\alpha}(n)|^2
\ll
x(\log x)^{|\alpha|^2}
\prod_{p}\left(1+O(p^{-\frac{3}{2}})\right)
\ll
x(\log x)^{|\alpha|^2}.
\]
Therefore, taking $C\ge2$ larger than some constant depending only on $\alpha$,
we find that $\tau_{\alpha}(n)$ satisfies \cref{V1}, \cref{V2}, and \cref{V3}
with the same constant $C$. For the estimate \cref{tau_condition:L1},
we proceed similarly to obtain
\begin{align}
\sum_{n\le x}\frac{|\tau_{\alpha}(n)|}{n}
&\le
\prod_{p\le x}
\left(1+\frac{|\tau_{\alpha}(p)|}{p}+\frac{|\tau_{\alpha}(p^2)|}{p^2}+\cdots\right)\\
&\le
\prod_{p\le x}
\left(1+\frac{|\alpha|}{p}\right)
\left(1+\frac{|\tau_{\alpha}(p^2)|}{p^2}+\frac{|\tau_{\alpha}(p^3)|}{p^3}
+\cdots\right)
\ll
(\log x)^{|\alpha|}.
\end{align}
This completes the proof.
\end{proof}

%%%%%%%%%%%%%%%%%%%%%%%%%%%%%%%%%%%%%%%%
\begin{proof}[Proof of \cref{Thm:BP_completed}]
By \cref{Thm:BP}, it suffices to prove
\[
\sum_{n\le y}\frac{v(n)}{n}\psi\left(\frac{x}{n}\right)
\ll(\log x)^{\frac{2|\alpha|}{3}}(\log\log x)^{\frac{|\alpha|}{3}}
\]
for sufficiently large $x$.
By \cref{BP:v}, we have
\[
v(n)=\sum_{dm=n}\tau_{\alpha}(d)b(m).
\]
Therefore,
\begin{equation}
\label{BP_completed:decompose}
\sum_{n\le y}\frac{v(n)}{n}\psi\left(\frac{x}{n}\right)
=
\sum_{m\le y}\frac{b(m)}{m}
\sum_{d\le y/m}\frac{\tau_{\alpha}(d)}{d}\psi\left(\frac{x/m}{d}\right).
\end{equation}
We apply \cref{Thm:v_thm} to the inner sum.
We have already checked in \cref{Lem:tau_condition}
that $\tau_{\alpha}(n)$ satisfies the assumptions on $v(n)$ of \cref{Thm:v_thm}
with $\kappa=|\alpha|$. For the assumption on $y$, we have
\[
\frac{y}{m}
=\frac{x}{m}\exp(-(\log x)^{\frac{1}{6}})
\le\frac{x}{m}\exp\left(-\left(\log\frac{x}{m}\right)^{\frac{1}{6}}\right)
\]
so this assumption is satisfied with $\theta=1/6$.
Also, we have $x/m\ge x/y>4$ for sufficiently large $x$.
Thus, by \cref{Thm:v_thm},
\[
\sum_{d\le y/m}\frac{\tau_{\alpha}(d)}{d}\psi\left(\frac{x/m}{d}\right)
\ll
\left(\log\frac{x}{m}\right)^{\frac{2|\alpha|}{3}}
\left(\log\log\frac{x}{m}\right)^{\frac{|\alpha|}{3}}
\ll
\left(\log x\right)^{\frac{2|\alpha|}{3}}
\left(\log\log x\right)^{\frac{|\alpha|}{3}}.
\]
By substituting this estimate into \cref{BP_completed:decompose}
and using the convergence of \cref{BP:b},
\begin{align}
\sum_{n\le y}\frac{v(n)}{n}\psi\left(\frac{x}{n}\right)
&\ll
\left(\log x\right)^{\frac{2|\alpha|}{3}}
\left(\log\log x\right)^{\frac{|\alpha|}{3}}
\sum_{m\le y}\frac{|b(m)|}{m}\\
&\ll
\left(\log x\right)^{\frac{2|\alpha|}{3}}
\left(\log\log x\right)^{\frac{|\alpha|}{3}}.
\end{align}
This completes the proof.
\end{proof}

%%%%%%%%%%%%%%%%%%%%%%%%%%%%%%%%%%%%%%%%
\section{Examples}
\label{Section:examples}
We now give some examples of \cref{Thm:BP_completed}.
For most of the examples below, the preparation has been already done
by Balakrishnan and P\'etermann~\cite{Balakrishnan_Petermann},
so we just refer \cite{Balakrishnan_Petermann} to check
the hypothesis of \cref{Thm:BP} for such examples.
We also include the singular series of the Goldbach conjecture
\begin{equation}
\label{singular_series}
\mathfrak{S}(n)
:=
\left\{
\begin{array}{>{\displaystyle}c>{\displaystyle}l}
\mathfrak{S}_2\prod_{\substack{p\mid n\\p>2}}\left(\frac{p-1}{p-2}\right)
&(\text{if $n$ is even}),\\[4mm]
0&(\text{if $n$ is odd}),\\
\end{array}
\right.\ \ 
\mathfrak{S}_2
:=
2\prod_{p>2}\left(1-\frac{1}{(p-1)^2}\right),
\end{equation}
as an example which was not mentioned in \cite{Balakrishnan_Petermann}.
This example was considered by Friedlander and Goldston~\cite{Friedlander_Goldston}
and by Languasco~\cite{Languasco1,Languasco2}.
We introduce a multiplicative function
\begin{equation}
\label{mini_singular_series}
\mathfrak{s}(n)
:=
\prod_{\substack{p\mid n\\p>2}}\left(\frac{p-1}{p-2}\right)
\end{equation}
in order to deal with the singular series~\cref{singular_series}.

%%%%%%%%%%%%%%%%%%%%%%%%%%%%%%%%%%%%%%%%
\begin{lemma}
\label{Lem:check_h}
For any complex number $\alpha$,
we have
\begin{align}
\sum_{n=1}^{\infty}\frac{1}{n^s}\left(\frac{\sigma(n)}{n}\right)^{\alpha}
&=
\zeta(s)\zeta(s+1)^{\alpha}f_{\alpha}^{(1)}(s+1),\\
\sum_{n=1}^{\infty}\frac{1}{n^s}\left(\frac{n}{\phi(n)}\right)^{\alpha}
&=
\zeta(s)\zeta(s+1)^{\alpha}f_{\alpha}^{(2)}(s+1),\\
\sum_{n=1}^{\infty}\frac{1}{n^s}\left(\frac{\sigma(n)}{\phi(n)}\right)^{\frac{\alpha}{2}}
&=
\zeta(s)\zeta(s+1)^{\alpha}f_{\alpha}^{(3)}(s+1),
\end{align}
and
\[
\sum_{n=1}^{\infty}\frac{\mathfrak{s}(n)^{\alpha}}{n^s}
=
\zeta(s)\zeta(s+1)^{\alpha}f_{\alpha}^{(4)}(s+1),
\]
where $f_{\alpha}^{(i)}(s)$ $(1\le i\le 4)$
are Dirichlet series convergent for $\sigma>1/2$.
\end{lemma}
%%%%%%%%%%%%%%%%%%%%%%%%%%%%%%%%%%%%%%%%
\begin{proof}
For the first two examples, see Lemma~5.1 and Lemma~5.2
of \cite[p.59--60]{Balakrishnan_Petermann}. In \cite{Balakrishnan_Petermann},
real $\alpha$ is mainly considered, but there is no difficulty to obtain
the same result for complex $\alpha$.
The other two examples can be dealt with in the same way.
For the ease of the reader, we prove the third example,
for which the special case $\alpha=2$ is stated in \cite[p.40, (3)]{Balakrishnan_Petermann}.

Let
\[
Z(s)
=\sum_{n=1}^{\infty}\frac{1}{n^s}\left(\frac{\sigma(n)}{\phi(n)}\right)^{\frac{\alpha}{2}},\quad
\zeta(s)^{-1}Z(s)
=\sum_{n=1}^{\infty}\frac{v(n)}{n^{s+1}},
\]
\[
f_{\alpha}^{(3)}(s+1)
=\zeta(s)^{-1}\zeta(s+1)^{-\alpha}Z(s)
=\sum_{n=1}^{\infty}\frac{b(n)}{n^{s+1}}
\]
for $\sigma>1$. Then $v(n)$ and $b(n)$ become multiplicative functions and
\begin{equation}
\label{check_h:v_b}
\frac{v(n)}{n}=\sum_{dm=n}\mu(d)\left(\frac{\sigma(m)}{\phi(m)}\right)^{\frac{\alpha}{2}},\quad
b(n)=\sum_{dm=n}\tau_{-\alpha}(d)v(m).
\end{equation}
For every prime $p$ and every integer $\nu\ge1$, we have an identity
\[
\left(\frac{\sigma(p^{\nu})}{\phi(p^{\nu})}\right)^{\frac{\alpha}{2}}
=
\left(\frac{1+p^{-1}+\cdots+p^{-\nu}}{1-p^{-1}}\right)^{\frac{\alpha}{2}}.
\]
By using the binomial expansion, we obtain
\[
\left(\frac{\sigma(p)}{\phi(p)}\right)^{\frac{\alpha}{2}}
=
\left(1+\frac{2}{p-1}\right)^{\frac{\alpha}{2}}
=
1+\binom{\frac{\alpha}{2}}{1}\frac{2}{p-1}+O\left(\frac{1}{p^2}\right)
=
1+\frac{\alpha}{p}+O\left(\frac{1}{p^2}\right)
\]
and
\[
\left(\frac{\sigma(p^{\nu})}{\phi(p^{\nu})}\right)^{\frac{\alpha}{2}}
=
\frac{(1+p^{-1}+\cdots+p^{-(\nu-1)})^{\frac{\alpha}{2}}+O(p^{-\nu})}{(1-p^{-1})^{\frac{\alpha}{2}}}
=
\left(\frac{\sigma(p^{\nu-1})}{\phi(p^{\nu-1})}\right)^{\frac{\alpha}{2}}
+
O\left(\frac{1}{p^{\nu}}\right),
\]
where the implicit constant depends on $\alpha$.
Therefore, by \cref{check_h:v_b},
\begin{align}
\frac{v(p)}{p}
&=\left(\frac{\sigma(p)}{\phi(p)}\right)^{\frac{\alpha}{2}}-1
=\frac{\alpha}{p}+O\left(\frac{1}{p^2}\right),\\
\frac{v(p^{\nu})}{p^{\nu}}
&=
\left(\frac{\sigma(p^{\nu})}{\phi(p^{\nu})}\right)^{\frac{\alpha}{2}}
-
\left(\frac{\sigma(p^{\nu-1})}{\phi(p^{\nu-1})}\right)^{\frac{\alpha}{2}}
=
O\left(\frac{1}{p^{\nu}}\right)\quad\text{for $\nu\ge2$}.
\end{align}
By \cref{Lem:tau_explicit} and \cref{check_h:v_b},
\[
b(p)
=
v(p)+\tau_{-\alpha}(p)=O\left(\frac{1}{p}\right),\quad
b(p^{\nu})
\ll
\sum_{dm=p^{\nu}}|\tau_{-\alpha}(d)|
\ll_{\epsilon}
p^{\nu\epsilon}.
\]
Thus, for $\sigma>1/2$, by taking $\epsilon>0$
sufficiently small so that $\sigma-\epsilon>1/2$,
\begin{align}
\sum_{n=1}^{\infty}\frac{|b(n)|}{n^{\sigma}}
&=
\prod_{p}
\left(1+\sum_{\nu=1}^{\infty}\frac{|b(p^{\nu})|}{p^{\nu\sigma}}\right)\\
&=
\prod_{p}
\left(1+O\left(\frac{1}{p^{\sigma+1}}
+\sum_{\nu=2}^{\infty}\frac{1}{p^{\nu(\sigma-\epsilon)}}\right)\right)\\
&=
\prod_{p}
\left(1+O\left(\frac{1}{p^{\sigma+1}}+\frac{1}{p^{2(\sigma-\epsilon)}}\right)\right)
<+\infty.
\end{align}
This completes the proof.
\end{proof}

%%%%%%%%%%%%%%%%%%%%%%%%%%%%%%%%%%%%%%%%
\begin{theorem}
\label{Thm:asymp}
For any complex number $\alpha$ and $x\ge4$, we have
\begin{align}
\label{asymp:1}
\sum_{n\le x}\left(\frac{\sigma(n)}{n}\right)^{\alpha}
&=
A^{(1)}x
+\sum_{r=0}^{[\Re\alpha]}A_{r}^{(1)}(\log x)^{\alpha-r}
+O\left((\log x)^{\frac{2|\alpha|}{3}}(\log\log x)^{\frac{|\alpha|}{3}}\right),\\
\label{asymp:2}
\sum_{n\le x}\left(\frac{n}{\phi(n)}\right)^{\alpha}
&=
A^{(2)}x
+\sum_{r=0}^{[\Re\alpha]}A_{r}^{(2)}(\log x)^{\alpha-r}
+O\left((\log x)^{\frac{2|\alpha|}{3}}(\log\log x)^{\frac{|\alpha|}{3}}\right),\\
\label{asymp:3}
\sum_{n\le x}\left(\frac{\sigma(n)}{\phi(n)}\right)^{\frac{\alpha}{2}}
&=
A^{(3)}x
+\sum_{r=0}^{[\Re\alpha]}
A_{r}^{(3)}(\log x)^{\alpha-r}
+O\left((\log x)^{\frac{2|\alpha|}{3}}(\log\log x)^{\frac{|\alpha|}{3}}\right),\\
\label{asymp:4}
\sum_{\substack{n\le x\\[0.3mm]n\colon\text{\upshape even}}}\mathfrak{S}(n)^{\alpha}
&=
A^{(4)}x
+\sum_{r=0}^{[\Re\alpha]}
A_{r}^{(4)}(\log x)^{\alpha-r}
+O\left((\log x)^{\frac{2|\alpha|}{3}}(\log\log x)^{\frac{|\alpha|}{3}}\right),
\end{align}
where $(A_{r}^{(i)})$ are complex coefficients computable
from the Laurent expansions of the Dirichlet series in \cref{Lem:check_h} at $s=1$
and the implicit constants depend only on $\alpha$.
\end{theorem}
%%%%%%%%%%%%%%%%%%%%%%%%%%%%%%%%%%%%%%%%
\begin{proof}
The first three asymptotic formulas immediately follow by \cref{Thm:BP_completed} and \cref{Lem:check_h}.
For the last asymptotic formula, by \cref{singular_series} and \cref{mini_singular_series},
we obtain
\[
\sum_{\substack{n\le x\\[0.3mm]n\colon\text{\upshape even}}}\mathfrak{S}(n)^{\alpha}
=
\mathfrak{S}_2^{\alpha}
\sum_{2n\le x}\mathfrak{s}(2n)^{\alpha}
=
\mathfrak{S}_2^{\alpha}
\sum_{n\le x/2}\mathfrak{s}(n)^{\alpha}
\]
since $\mathfrak{s}(2^{\nu})=1$. By \cref{Thm:BP_completed} and \cref{Lem:check_h},
we obtain
\[
\mathfrak{S}_2^{\alpha}\sum_{n\le x/2}\mathfrak{s}(n)^{\alpha}
=
\tilde{A}^{(4)}x
+\sum_{r=0}^{[\Re\alpha]}
\tilde{A}_{r}^{(4)}\left(\log\frac{x}{2}\right)^{\alpha-r}
+O\left((\log x)^{\frac{2|\alpha|}{3}}(\log\log x)^{\frac{|\alpha|}{3}}\right).
\]
On inserting the expansion
\begin{align}
\left(\log\frac{x}{2}\right)^{\alpha-r}
&=
(\log x)^{\alpha-r}
\left(1-\frac{\log 2}{\log x}\right)^{\alpha-r}\\
&=
\sum_{\ell=0}^{\infty}
(-1)^{\ell}\binom{\alpha-r}{\ell}(\log 2)^{\ell}(\log x)^{\alpha-r-\ell}\\
&=
\sum_{\ell=0}^{[\Re\alpha]-r}
(-1)^{\ell}\binom{\alpha-r}{\ell}(\log 2)^{\ell}(\log x)^{\alpha-r-\ell}
+O(1),
\end{align}
we obtain the last asymptotic formula.
\end{proof}

%%%%%%%%%%%%%%%%%%%%%%%%%%%%%%%%%%%%%%%%
\begin{remark}
To the knowledge of the author,
the asymptotic formulas in \cref{Thm:asymp}
provide the error term estimates better than the previous results except the cases
\begin{itemize}
\item $\alpha=+1$ of the formula \cref{asymp:1} (by Walfisz~\cite{Walfisz}),
\item $\alpha=+1$ of the formula \cref{asymp:2} (by Sitaramachandrarao~\cite{Sitaramachandrarao}),
\item $\alpha=-1$ of the formula \cref{asymp:2} (by Liu~\cite{HQ_Liu}),
\item $\alpha=+1$ of the formula \cref{asymp:3} (probably well known to the experts),
\item $\alpha=+1$ of the formula \cref{asymp:4} (by Friedlander and Goldston~\cite{Friedlander_Goldston}).
\end{itemize}
These cases except Liu's result provide the error term estimate
\[
O((\log x)^{\frac{2}{3}})
\]
which is better than \cref{Thm:asymp}.
The source of this phenomenon
is the fact that there is no log-power in \cref{Lem:Walfisz}
similar to $(\log Q)^{6C+6}$ in \cref{Lem:v}.
\end{remark}

%%%%%%%%%%%%%%%%%%%%%%%%%%%%%%%%%%%%%%%%
\section*{Acknowledgements}
The author would like to deeply thank
Pieter Moree, Sumaia Saad Eddin, and Alisa Sedunova
for motivating the author on this topic through the joint work~\cite{MSSS}.
The author would like to express his sincere gratitude to Kohji Matsumoto
for his continued support and helpful advice.
The author is also indebted to
Hirotaka Kobayashi, Kota Saito, and Wataru Takeda
for having a seminar to check the result and giving valuable advice.
Finally, the author wishes to express his thanks
to Isao Kiuchi and Hiroshi Mikawa for their comments.
This work was supported by Grant-in-Aid for JSPS Research Fellow
(Grant Number: JP16J00906).

%%%%%%%%%%%%%%%%%%%%%%%%%%%%%%%%%%%%%%%%

\vspace{3mm}

\begin{flushleft}
{\small
{\sc
Graduate School of Mathematics, Nagoya University,\\
Chikusa-ku, Nagoya 464-8602, Japan.
}

{\it E-mail address}: {\tt m14021y@math.nagoya-u.ac.jp}
}
\end{flushleft}
\end{document}